\documentclass[10pt]{amsart}
\usepackage{amsmath}
\usepackage{paralist}
\usepackage{graphicx} %% add this and next lines if pictures should be in esp format
\usepackage{epstopdf}
\usepackage[colorlinks=true]{hyperref}
\usepackage{float}
\usepackage{color,comment}
\usepackage{caption}
\usepackage{wrapfig}
\usepackage{multirow}
\usepackage{tabularx}
\usepackage{booktabs}
\usepackage{orcidlink}
\usepackage{subfigure}
\usepackage[toc]{appendix}
\newtheorem{theorem}{Theorem}[section]

\newtheorem{conjecture}{Conjecture}

\theoremstyle{definition}
\newtheorem{definition}[theorem]{Definition}
\newtheorem{remark}{Remark}[section]

\DeclareMathOperator{\Tr}{tr}

\begin{document}
\title[Dual Fear with Prey Aggregation]{Dual Fear Phenomenon in an Eco-Epidemiological Model with Prey Aggregation}

\author[Antwi-Fordjour, Westmoreland, Bearden]{}
\maketitle

\centerline{\scshape Kwadwo Antwi-Fordjour$^{*1} {\orcidlink{0000-0002-4961-0462}}$, Sarah P. Westmoreland$^1$,   Kendall H. Bearden$^1$}

\vspace{.7cm}
{\centerline{1) Department of Mathematics and Computer Science,}
 \centerline{ Samford University,}
 \centerline{ Birmingham, AL 35229, USA.}
}

%\bigskip
\vspace{.7cm}
\centerline{ *Corresponding author's email: kantwifo@samford.edu;}
\smallskip
\centerline{ Contributing authors: swestmor@samford.edu; kbearde2@samford.edu;}

\begin{abstract}
This study presents a thorough analysis of an eco-epidemiological model that integrates infectious diseases in prey, prey aggregation, and the {\em dual} fear effect induced by predators. We establish criteria for determining the existence of equilibrium points, which carry substantial biological significance. We establish  the conditions for the occurrence of Hopf, saddle-node, and transcritical bifurcations by employing fear parameters as key bifurcation parameters. Furthermore, through numerical simulations, we demonstrate the occurrence of multiple zero-Hopf (ZH) and saddle-node transcritical (SNTC) bifurcations around the endemic steady states by varying specific key parameters across the two-parametric plane. We demonstrate that the introduction of predator-induced fear, which hinders the growth rate of susceptible prey, can lead to the finite time extinction of an initially stable susceptible prey population. Finally, we discuss management strategies aimed at regulating disease transmission, focusing on fear-based interventions and \emph{selective predation} via predator attack rate on infectious prey.  
\end{abstract}

\section{Introduction}
\noindent In ecological systems, prey aggregation represents a remarkable anti-predator strategy, offering protection to prey species against predators. To model this phenomenon, an exponential functional response, $g(S) = d_0 S^r$, with $0 < r < 1$, was introduced by Rosenzweig \cite{R71}, where $S = S(t)$ denotes the population density of the prey, and $d_0$ represents the predation rate. Various studies have explored the implications of prey aggregation in prey-predator interactions, including symbiotic, competitive, and predator-prey models \cite{M18, A11}. For instance, Braza \cite{B12} investigated a predator-prey model utilizing the square root functional response $g(S) = d_0 S^{1/2}$, proposed by Gauss \cite{G34}, indicating a strong herd behavior where predators interact with the prey at the herd's periphery. Additional approaches for modeling prey aggregation (or herd behavior) have been discussed in the literature \cite{B11, V13, V18, Y13, Z21, Y20}, offering a comprehensive perspective on this intriguing ecological phenomenon.

Additionally, the concept of finite time extinction (FTE) has significant implications for ecosystem management. Notably, the functional response $g(S)$, with $r < 1$, exhibits intriguing non-smooth characteristics when $S = 0$, leading to complex dynamics. This functional response enables the prey species to go extinct in finite time, followed by the exponential decay of the predator population to zero in infinite time \cite{APB20}. Non-smooth functional responses, also known as power incidence functions, have been extensively analyzed in susceptible-infective models, where host species may potentially face finite time extinction \cite{F18}. Understanding the dynamics of finite time extinction enhances our ability to manage and conserve natural resources effectively.

Eco-epidemiology is an interdisciplinary field that combines concepts from ecology and epidemiology to study the interactions between infectious diseases and ecological systems. It focuses on understanding how ecological factors influence the transmission and dynamics of infectious diseases, and how disease dynamics, in turn, impact the populations and communities in ecological systems. Kooi and Venturino \cite{KV16} explored an eco-epidemic predator-prey model that showcases prey herd behavior and abandoned infected prey, further contributing to our understanding of these intricate dynamics.  Gupta and Dubey \cite{G22} explored an eco-epidemic model wherein the predator receives additional food. Their findings revealed that the surplus energy obtained from the additional food can potentially result in disease eradication, even under higher infection rates. For further readings on eco-epidemiological models, please see \cite{V16, SS20, D21, SP22} and references therein.

To incorporate the profound influence of fear in predator-prey interactions, several notable studies have contributed to our understanding of this complex phenomenon \cite{Z11, H14, S16}. In 2016, Wang et al. \cite{Wang16} proposed the first mathematical predator-prey model integrating fear effects, particularly focusing on the Holling type II functional response. They intriguingly observed that elevated fear levels can stabilize the predator-prey dynamics, offering insights into the intricate balance between fear-induced behaviors and the functional response. Recently, Antwi-Fordjour et al. \cite{APTW23} made a compelling observation that increasing the strength of fear in a predator-prey model with prey herd behavior and mutual interference could lead to the extinction of the prey population from a coexistence zone in finite time. This finding highlights the critical implications of fear in altering the long-term dynamics of ecological systems. 

Furthermore, Sha et al. \cite{SSJ19} also made noteworthy contributions by investigating an eco-epidemiological model with  disease in the prey species and fear of predators. They observed fear-induced backward bifurcation, bistability, oscillations and the occurrence of chaos, adding depth to the comprehension of predator-prey dynamics. Moreover, Hossain et al. \cite{HPS20} undertook a comprehensive study on the influence of fear in an eco-epidemiological model with a general disease transmission function for which mass action, standard incidence, and saturation laws are specific cases. Their investigation provided valuable insights into how fear can eliminate chaotic oscillations which influence the demographic structure and stability. For further enriching the discourse on the impact of fear in predator-prey dynamics, the following works  offer invaluable perspectives \cite{S22, KA22, VA22}. These collective endeavors contribute to a comprehensive understanding of fear's profound impact on ecological systems, underscoring the need for further exploration and research in this captivating field.

The main aim of this paper is to study the effect of dual fear of predators in an eco-epidemiological model. We assume a split in the prey's population into susceptible and infectious populations. We assume that once the prey is infected with a disease, recovery is impossible. We posit that the susceptible prey population exhibits an anti-predator behavior by aggregating when they encounter predators. Following the work of Shar et al. \cite{SSJ19}, the dual fear effect considered in this paper are:
\begin{itemize}
\item[(i)] fear of predators that reduces the birth rate of the susceptible prey and 
\item[(ii)] fear of predators that lowers the transmission of disease.
\end{itemize}
 We assume the prey population aggregate via the Rosenzweig functional response $g(S)$. The core objectives of our research are to address the following inquiries:
\begin{enumerate}
 \item Do fear parameters acting as key bifurcation parameters induce a plethora of rich dynamical behaviors in the eco-epidemiological model incorporating prey aggregation and a dual fear effect?
\item How does selective predation influence the time evolution of population densities in regulating disease transmission?
\item What ecological factors or predator-prey dynamics lead to the shift from stable endemic state to finite time extinction of the susceptible prey population when a positive value of parameter is introduced?
\end{enumerate}

The present study is organized as follows: In Section \ref{sec:model formulation}, we introduce an eco-epidemiological model with disease and aggregation in the prey population. Section \ref{sec:preliminary results} presents preliminary findings, including nonnegativity and boundedness. In Section \ref{sec:equilibria and stability}, we explore the existence of biologically significant equilibrium points and discuss their stability. We conduct comprehensive co-dimension one bifurcation analysis for our proposed model in Section \ref{sec:bifucation analysis}. Section \ref{sec:FTE} focuses on the finite time extinction of the susceptible prey population, presenting relevant results. In Section \ref{sec:Disease management}, we discuss disease management strategies aimed at controlling the spread of disease. Numerical results are extensively presented in various sections to complement our theoretical findings and support our conjectures. These numerical simulations provide a practical validation of the proposed eco-epidemiological model and offer a visual representation of its rich dynamics. We summarize the conclusions drawn from this research in Section \ref{sec:conclusion}. Furthermore, we explore the practical implications of our findings within ecological and epidemiological contexts, elucidating promising directions for future research in this domain.

\section{The Mathematical Model}\label{sec:model formulation}
\noindent In our ecological model, we make several biologically motivated assumptions regarding the dynamics of the prey and predator populations and the role of fear of predation in disease spread:
\begin{itemize}
\item[(a)] Prey Population Split: The prey population is divided into two distinct classes: susceptible and infectious. This categorization allows us to capture the transmission dynamics of the disease within the prey population accurately.

\item[(b)] Transmission of Infection: Susceptible prey become infected by direct contact with infected prey individuals. We assume that the transmission occurs through mechanisms such as physical contact or other modes of disease transmission that are common in ecological systems.

\item[(c)] Impact of Infection: The infectious prey, once infected, are assumed to be in a weakened state, rendering them unable to reproduce or compete effectively for vital resources. This assumption reflects the reduced fitness and compromised physiological state often observed in infected individuals.

\item[(d)] Fear of Predation: We incorporate the concept of fear of predation into our model. This fear reduces the interactions between infected and susceptible individuals, leading to a significant decrease in the spread of the disease. The presence of predators induces changes in prey behavior, resulting in reduced contact rates and, consequently, diminished disease transmission. Based on experimental evidence, we can assume that fear of predators leads to a decline in the prey's growth rate.

\item[(e)] Irreversibility of Infection: We assume that once a prey individual becomes infected, they cannot recover from the infection. This assumption aligns with certain infectious diseases where recovery is not possible or is highly unlikely within the time frame considered in the model.

\item[(f)] Carrying Capacity: In our proposed model, we incorporate the biological assumption that both the susceptible and infected prey populations contribute to the carrying capacity. This means that the total capacity of the ecosystem to support prey is determined not only by the size of the susceptible population but also by the presence of the infected individuals.

\item[(g)] Predator-Prey Dynamics: We assume the following interactions in our proposed model:
\begin{itemize}
\item[(i)] Logistic Growth: In the absence of predation, the susceptible prey population grows logistically.

\item[(ii)] Prey Aggregation: We assume the prey population aggregates when they encounter predators to form a defense mechanism. This is modeled by $S^r$, where $r\in (0,1)$ is the aggregating constant.

\item[(iii)] Mass Action (or Holling type I): The interaction between infected prey and predator populations is characterized by the ecological principle known as mass action.  With this type of functional response, the consumption rate of the predator increases linearly with an increase in the density of its prey up to a certain point and then begin to saturate since the maximum rate of consumption has been attained by the predator. Also, the interaction between the susceptible prey and infected prey is modeled via a modified mass action. 
\item[(iv)] Natural Death Rates: We incorporate natural death rates for susceptible prey, infected prey, and predators, capturing the inherent mortality in ecological systems. This accounts for population turnover and enables a realistic representation of ecological processes, birth rates, disease transmission, and predation dynamics.
\end{itemize}
\end{itemize}
These assumptions provide a biologically plausible framework for studying the eco-epidemiological dynamics of the prey population, taking into account disease transmission, predator-prey interactions, and the role of fear of predation in the spread of the disease.

Let $f(k_i,P)=\frac{1}{1+k_iP}$ be the fear function where $i=$1 or 2. This decreasing function  was proposed by Wang et al. \cite{Wang16} . The fear function $f(k_i,P)$ possesses several biological properties that characterize its behavior within the context of the model:
\begin{itemize}
\item[(i)] In the absence of fear ($k_i=0$), the fear function yields:
\begin{equation*}
f(0,P) = \dfrac{1}{1+0\cdot P}=1.
\end{equation*}
This implies that when there is no fear, the interaction between individuals is not influenced, and the fear function does not affect the system.

\item[(ii)] When there are no predators ($P=0$), the fear function becomes:
\begin{equation*}
f(k_i,0)= \dfrac{1}{1+k_i\cdot 0}=1.
\end{equation*}
This indicates that in the absence of predators, the fear function does not alter the interaction dynamics between individuals.

\item[(iii)]  As the fear level increases ($k_i\rightarrow \infty$), the fear function approaches the limit:
\begin{equation*}
\lim\limits_{k_i\rightarrow \infty}f(k_i,P)= \lim\limits_{k_i\rightarrow \infty} \dfrac{1}{1+k_iP}=0.
\end{equation*}
This signifies that when fear is exceedingly high, it leads to a significant reduction in the interaction among individuals, ultimately suppressing the birth rate and disease spread.

\item[(iv)]  When the predator population is abundant ($P\rightarrow \infty$), the fear function tends to zero, causing the birthrate or disease transmission rate to decline:
\begin{equation*}
\lim\limits_{P\rightarrow \infty}f(k_i,P)= \lim\limits_{P\rightarrow \infty} \dfrac{1}{1+k_iP}=0.
\end{equation*}
This implies that in the presence of a large predator population, the fear function approaches zero, resulting in a collapse of the birthrate or disease transmission rate.
\end{itemize}
These properties of the fear function capture various biological and epidemiological aspects of the model, such as the impact of fear levels, predator presence, and population dynamics on the transmission and spread of the disease.

We develop a susceptible-infectious-predator ($SIP$) model, representing the dynamics of three populations: susceptible prey ($S$), infectious prey ($I$), and predators ($P$). In this model, we consider the following relationships:

\begin{align}
\nonumber \frac{dS}{dt} &= \frac{b_0S}{1 + k_1P}\left(1-\frac{S+I}{K}\right) - a_0S - d_0S^rP - \frac{e_0SI}{1 + k_2P}=G_1(S,I,P) \\
   \frac{dI}{dt} &= -a_1I + \frac{e_0SI}{1+k_2P} - d_1IP =G_2(S,I,P)  \label{Mainsystem}  \\
\nonumber     \frac{dP}{dt} &= -a_2P + d_2S^rP + d_3IP=G_3(S,I,P)
\end{align}\\

\begin{table}[htbp]
    \centering
 \caption{Biological description of parameters from model \eqref{Mainsystem}. All parameters are assumed to be positive.}
    \label{tab:description}    
\begin{tabular}{| c | l |}
 \hline 
 Parameters & Biological Description \\
 \hline
 $b_0$ & Natural birth rate \\  
 $r$ & Aggregating constant\\
 $e_0$ & Disease transmission rate\\
 $K$ & Carrying capacity \\
 $a_0$ & Susceptible prey natural death rate \\
 $a_1$ & Infectious prey natural death rate \\
 $a_2$ & Predator natural death rate \\
 $d_0$ & Attack rate of the predator on the susceptible prey\\
 $d_1$ & Attack rate of the predator on infectious prey\\
 $d_2$ & Biomass conversion efficiency from susceptible prey to predator \\
 $d_3$ & Biomass conversion efficiency from infected prey to predator \\
 $k_1$ & Level of fear that suppresses growth rate of susceptible prey \\
 $k_2$ & Level of fear that reduces disease transmission \\
 \hline
\end{tabular}
\end{table}

\noindent The biological description of parameters used in model \eqref{Mainsystem} are listed in Table \ref{tab:description}.  To the best of our knowledge, this is the first eco-epidemiological model to consider dual fear effect and susceptible prey aggregation via Rosenzweig functional response function (i.e., $S^r$) as depicted in model \eqref{Mainsystem}. In the next section, we shall investigate the well-behavedness of model \eqref{Mainsystem}.

%%%====
\section{Preliminary Results}\label{sec:preliminary results}
\noindent To ensure the validity of model \eqref{Mainsystem}, we will demonstrate its nonnegativity, ensuring that all populations involved have nonnegative values, which aligns with biological feasibility. Additionally, we will establish the boundedness of the model, indicating that each population remains within a finite upper limit, further supporting its biological realism.
\subsection{Nonnegativity}
\begin{theorem}\label{thm:nonnegativity}
All solutions $(S(t), I(t), P(t))$ of the model \eqref{Mainsystem} are nonnegative for all $t\geq 0$.
\end{theorem}

\begin{proof}
The right hand side of model \eqref{Mainsystem} is continuous and locally non-smooth function of the dependent variable $t$. We obtain the following after integration.
\begin{align*}
S(t)&=S(0)\text{exp}\left(\displaystyle\int_0^t \Big[\frac{b_0}{1 + k_1P}\left(1-\frac{S+I}{K}\right) - a_0 - d_0S^{r-1}P - \frac{e_0I}{1 + k_2P}\Big] ds\right)\geq 0 \\
I(t)&=I(0)\text{exp}\left(\displaystyle\int_0^t \Big[-a_1 + \frac{e_0S}{1+k_2P} - d_1P\Big]   ds\right)\geq 0\\
P(t)&=P(0)\text{exp}\left(\displaystyle\int_0^t \Big[-a_2 + d_2S^r + d_3I\Big] ds\right)\geq 0
\end{align*}

\noindent Therefore, all solutions initiating from the interior of the first octant remain in it for all future time.
\end{proof}

\subsection{Boundedness}

\begin{theorem}
All the solutions $(S(t), I(t), P(t))$ of the model \eqref{Mainsystem} with positive initial conditions are uniformly bounded if $d_2 < d_0$ and $d_3 < d_1$.
\end{theorem}

%\emph{\textbf{Theorem 3.2}: Given the following conditions,

\begin{proof}
Let us define the function $Q\left(S(t),I(t),P(t)\right) = S(t) + I(t) + P(t)$. Then
\begin{align*}
    \dfrac{dQ}{dt} &= \dfrac{dS}{dt} + \dfrac{dI}{dt} + \dfrac{dP}{dt}\\
    &= \dfrac{b_0S}{1 + k_1P}\left(1 - \dfrac{S + I}{K}\right) - a_0S - d_0S^rP - \dfrac{e_0SI}{1+k_2P} -a_1I + \dfrac{e_0SI}{1+k_2P} - d_1IP\\
    & -a_2P + d_2S^rP + d_3IP\\
    &=\frac{b_0S}{1 + k_1P}\left(1 - \frac{S + I}{K}\right) - a_0S - S^rP(d_0 - d_2)  -a_1I - IP(d_1 - d_3) -a_2P\\\end{align*}
  We assume that $d_2 < d_0$ and $d_3 < d_1$, we obtain the following inequality, 

\begin{align*}
   \frac{dQ}{dt} &\leq b_0S\left(1 - \frac{S}{K}\right) - a_0S -a_1I -a_2P
\end{align*}
Let $\xi$ be a positive constant, then
\begin{align*}
   \frac{dQ}{dt} + \xi Q &\leq b_0S\left(1 - \frac{S}{K}\right) - a_0S -a_1I -a_2P + \xi(S + I + P) \\
   &\leq S\left(b_0\left(1 - \frac{S}{K}\right) - a_0 + \xi\right) -I(a_1 - \xi) -P(a_2 - \xi)
\end{align*}
 We assume $\xi \leq \min \{a_1, a_2\}$, thus we obtain
\begin{align*}
    \frac{dQ}{dt} + \xi Q  &\leq b_0S\left(\left(1 - \frac{S}{K}\right) - \frac{a_0 - \xi}{b_0}\right) \\
    &= b_0\left(\left(1-\frac{a_0-\xi}{b_0}\right)S-\frac{S^2}{K}\right)\\
%    &=\frac{b_0}{K}\left(K\left(1-\frac{a_0-\xi}{b_0}\right)S-S^2\right)
\end{align*}
We assume $\xi>a_0-b_0$, then
\begin{align*}
 \frac{dQ}{dt} + \xi Q & \leq \frac{b_0K}{4}\left(1-\frac{a_0-\xi}{b_0}\right)^2
\end{align*}
Taking $W=\frac{b_0K}{4}\left(1-\frac{a_0-\xi}{b_0}\right)^2>0$ and applying the theorem on differential inequality, we get
\[0\leq Q(S(t),I(t),P(t))\leq \frac{W(1-e^{-\xi t})}{\xi}+Q(S(0),I(0),P(0))e^{-\xi t},\]
which implies
\[\limsup_{t\rightarrow \infty}Q(S(t),I(t),P(t))\leq \frac{W}{\xi}\]

\noindent Hence, all the solutions of model \eqref{Mainsystem} which initiated in $\mathbb{R}_+^3$ are confined in 
\begin{equation*}
    \Gamma = \{(S, I, P) \in \mathbb{R}^3_+ : Q(S(t), I(t), P(t)) \leq \frac{W}{\xi} + \epsilon, \epsilon \in \mathbb{R}\}.
\end{equation*}

\end{proof}

\section{Equilibria and their Local Stability}\label{sec:equilibria and stability}
\noindent In this Section, we delve into the examination of biologically significant equilibrium points within model \eqref{Mainsystem}. We analyze their existence and further explore their stability properties. This analysis provides valuable insights into the dynamics of the system and sheds light on its long-term behavior.

\subsection{Biologically Significant Equilibrium Points}
The nonnegative equilibrium points of the model \eqref{Mainsystem} can be computed by solving the system
\begin{align}
   \label{susc=0} 0 &= S \left(\frac{b_0}{1 + k_1P}\left(1-\frac{S+I}{K}\right) - a_0 - d_0S^{r-1}P - \frac{e_0I}{1 + k_2P}\right) \\
    \label{inf=0} 0 &= I\left(-a_1 + \frac{e_0S}{1+k_2P} - d_1P\right) \\
   \label{pred=0} 0 &= P(-a_2 + d_2S^r + d_3I)
\end{align}
in $\mathbb{R}^3_+$. The biologically significant nonnegative equilibria are
\begin{itemize}
\item[(i)] Extinction equilibrium $E_0=(0,0,0)$.

\item[(ii)] Susceptible prey only $E_1=(S_1,0,0)$, where $S_1=K(1-\frac{a_0}{b_0})$. $E_1$ is feasible for $b_0>a_0$. 

\item[(iii)] Predator free $E_2=(S_2,I_2,0)$, where $S_2=\frac{a_1}{e_0}$ and $I_2= \frac{e_0 K \left(b_0-a_0\right)-a_1 b_0}{e_0 \left(b_0+e_0 K\right)}$. $E_2$ is biologically significant provided $a_0<b_0\left(1-\frac{a_1}{e_0 K}\right)$.

\item[(iv)] Infectious prey free $E_3=(S_3,0,P_3)$, where
\begin{equation}\label{cond:s3}
S_3=\left(\frac{a_2}{d_2}\right)^{1/r}
\end{equation}
 and $P_3$ is the positive root(s) of 
\begin{equation*}
 h_1 P^2+ h_2 P + h_3=0,
\end{equation*}
where
\[h_1= \left(-k_1d_0\left(\left(\frac{a_2}{d_2}\right)^{1/r}\right)^{r-1}\right)\]
\[h_2=\left(-d_0\left(\left(\frac{a_2}{d_2}\right)^{1/r}\right)^{r-1}-a_0k_1\right)\]
\[h_3=\left(b_0\left(1-\frac{1}{K}\left(\frac{a_2}{d_2}\right)^{1/r}\right) - a_0\right)\]
and
\begin{equation}\label{cond:P3}
P_3=\dfrac{-h_2+\sqrt{h^2_2-4h_1h_3}}{2h_1}=\frac{\frac{\left(\left(\frac{a_2}{d_2}\right){}^{1/r}\right){}^{-r} \left(\sqrt{B}-a_0 K k_1 \left(\frac{a_2}{d_2}\right){}^{1/r}\right)}{d_0 K}-1}{2 k_1}.
\end{equation}
 Also, \\
\noindent $B = K \left(C+D\right),$\\
where\\
$C=-2 d_0 k_1 \left(\left(\frac{a_2}{d_2}\right){}^{1/r}\right){}^{r+1} \left(2 b_0 \left(\left(\frac{a_2}{d_2}\right){}^{1/r}-K\right)+a_0 K\right)$,\\
$D=a_0^2 K k_1^2 \left(\frac{a_2}{d_2}\right){}^{2/r}+d_0^2 K \left(\left(\frac{a_2}{d_2}\right){}^{1/r}\right){}^{2 r}$.

\item[(v)] Coexistence $E_4=(S^*,I^*,P^*)$. 
The coexistence (or interior or endemic) equilibrium point(s) can be obtained from the nullclines (or isoclines) equation below:
\begin{align}
   \label{susc=S*} 0 &= \frac{b_0}{1 + k_1P}\left(1-\frac{S+I}{K}\right) - a_0 - d_0S^{r-1}P - \frac{e_0I}{1 + k_2P} \\
    \label{inf=I*} 0 &= -a_1 + \frac{e_0S}{1+k_2P} - d_1P \\
   \label{pred=P*} 0 &= -a_2 + d_2S^r + d_3I
\end{align}
\end{itemize}
where $S^*=\left(\frac{a_2-d_3I^*}{d_2} \right)^{1/r}$. $P^*$ is the positive root(s) of 
\[w_1P^2+w_2P+w_3=0\]
where
$w_1=d_1k_2,~w_2=d_1+a_1k_2,~w_3=a_1-\left(\dfrac{a_2-d_3I^*}{d_2}\right)^{1/r}$,
and $P^*=\frac{-w_2+\sqrt{w^2_2 - 4w_1w_3}}{2w_1}$.
Also, $I^*=-\frac{\left(k_2 P^*+1\right) \left(a_0 K S^* \left(k_1 P^*+1\right)+b_0 S^* (S^*-K)+d_0 K P^* \left(k_1 P^*+1\right) S^{*r}\right)}{S^* \left(b_0 \left(k_2 P^*+1\right)+e_0 K \left(k_1 P^*+1\right)\right)}$.

\subsection{Local Stability Analysis}
The Jacobian matrix for model \eqref{Mainsystem} at any equilibrium point can be expressed in the following form, which captures the interdependencies and interactions among the variables:
\[
\mathbb{J}=
\begin{bmatrix}
 J_{11}& 
     J_{12} & 
    J_{13}\\[1ex] 
  \ J_{21} & 
     J_{22} & 
     J_{23} \\[1ex]
  J_{31} & 
    J_{32} & 
     J_{33}
\end{bmatrix}
\]
where
\begin{align*}
   J_{11} &=\frac{b_0(K-2S-I)}{K(1 + k_1P)}-a_0 -rd_0S^{r-1}P - \frac{e_0I}{1+k_2P}\\
     J_{12}&= -\frac{b_0S}{K(1+k_1P)}-\frac{e_0S}{1+k_2P}<0\\
     J_{13} &= \frac{k_1b_0S\left(-K+S+I\right) }{K(1+k_1P)^2}-d_0S^r+\frac{k_2e_0SI}{(1+k_2P)^2}\\  
   J_{21} &= \frac{e_0I}{1+k_2P}>0\\
    J_{22} &= \frac{e_0S}{1+k_2P}-a_1 -d_1P\\
    J_{23} &=-\frac{k_2e_0SI}{(1+k_2P)^2}-d_1I<0\\
   J_{31} &= rd_2S^{r-1}P>0\\
    J_{32}&= d_3P>0\\
   J_{33} &= -a_2 +d_2S^r+d_3I
\end{align*}

\begin{remark}\label{remark: E_0}
The extinction equilibrium point $E_0$ cannot be studied with the method of linearization via the Jacobian matrix above due to the singularity caused by terms defined in $J_{11}$ and $J_{31}$ at the origin. This leads to non-uniqueness of solutions in backward time. We will omit the study of the extinction equilibrium point in this research. 
\end{remark}

The subsequent theorems demonstrate the local stability of the equilibrium points (i.e. susceptible prey only, predator free, infectious prey free, and coexistence), providing insights into their stability properties.

\begin{theorem}
The locally asymptotic stability of the equilibrium point $E_1$, representing only susceptible prey, is established under the conditions: $a_0 < b_0$, $a_1 > e_0K\left(1-\frac{a_0}{b_0}\right)$, and $a_2 > d_2 \left(K-\frac{a_0 K}{b_0}\right)^r$. These conditions ensure that the susceptible prey population remains stable in the absence of infectious prey and predators.
\end{theorem}

\begin{proof}
The Jacobian matrix evaluated at $E_1$ is given by
\[\mathbb{J}_{E_1}=
\begin{bmatrix}
 a_0-b_0 & 
   \frac{\left(a_0-b_0\right) \left(b_0+e_0 K\right)}{b_0}& 
   \frac{a_0 k_1 K \left(a_0-b_0\right)}{b_0}\\[1ex] 
    0 & 
   e_0 K \left(1-\frac{a_0}{b_0}\right)-a_1 & 
    0 \\[1ex]
    0 & 
    0 & 
    d_2 \left(K-\frac{a_0 K}{b_0}\right)^r-a_2
\end{bmatrix}\]

\noindent The eigenvalues obtained from the $\mathbb{J}_{E_1}$ are $\lambda_1= a_0-b_0,~\lambda_2= e_0 K \left(1-\frac{a_0}{b_0}\right)-a_1$ and $\lambda_3= d_2 \left(K-\frac{a_0 K}{b_0}\right)^r-a_2$. The equilibrium point $E_1$ will be locally asymptotically stable if all the eigenvalues are negative or have negative real parts. Thus, the model \eqref{Mainsystem} at $E_1$ is locally asymptotically stable if $a_0<b_0,~ a_1>e_0K\left(1-\frac{a_0}{b_0}\right)$ and $a_2> d_2 \left(K-\frac{a_0 K}{b_0}\right)^r$. The equilibrium point $E_1$ is unstable if at least one of the following three conditions is satisfied:\\
$a_0>b_0,~ a_1<e_0K\left(1-\frac{a_0}{b_0}\right)$ and $a_2< d_2 \left(K-\frac{a_0 K}{b_0}\right)^r$.
\end{proof}

\begin{theorem}\label{thm:stable E_2}
The local asymptotic stability of the predator-free equilibrium $E_2$ is established when the conditions $A_{33} < 0$, $A_{11} < 0$, and $A_{12}A_{21} < 0$ are satisfied. These conditions ensure that the absence of predators leads to a stable equilibrium state, where the predator population remains at zero and the prey populations exhibit stable dynamics.
\end{theorem}

\begin{proof}
The Jacobian matrix evaluated at $E_2$ is given by
\[\mathbb{J}_{E_2}=
\begin{bmatrix}
A_{11} & 
   A_{12}& 
   A_{13}\\[1ex] 
   A_{21} & 
   0 & 
  A_{23}\\[1ex]
    0 & 
    0 & 
    A_{33}
\end{bmatrix}\]
where
\begin{align*}
A_{11}&=-\dfrac{a_1b_0}{e_0K} \\
A_{12}&=-a_1 \left(\frac{b_0}{e_0 K}+1\right) \\
A_{13}&=\frac{a_1 \left(b_0 \left(k_1-k_2\right) \left(a_1-e_0 K\right)-a_0 \left(b_0 k_1+e_0 k_2 K\right)\right)}{e_0 \left(b_0+e_0 K\right)}-d_0 \left(\frac{a_1}{e_0}\right)^r\\
A_{21}&= \frac{e_0 K \left(b_0-a_0\right)-a_1 b_0}{\left(b_0+e_0 K\right)}\\
A_{23}&=  -\frac{\left(a_1 k_2-d_1\right) \left(e_0 K \left(a_0-b_0\right)+a_1 b_0\right)}{e_0 \left(b_0+e_0 K\right)} \\
A_{33}&=\frac{d_3 \left(e_0 K \left(b_0-a_0\right)-a_1 b_0\right)}{e_0 \left(b_0+e_0 K\right)}+d_2 \left(\frac{a_1}{e_0}\right)^r-a_2
\end{align*}
The eigenvalues of the Jacobian matrix at $E_2$ are $\lambda_1=A_{33}$ and the roots of the characteristic polynomial given by
\begin{equation*}
\lambda^2-A_{11}\lambda-A_{12}A_{21}=0.
\end{equation*}
For this characteristic polynomial, the roots are $\lambda_2+\lambda_3=A_{11}$ and  $\lambda_2\lambda_3=-A_{12}A_{21}$. Hence, $E_2$ is locally asymptotically stable if $A_{33}<0,~A_{11}<0$, and $A_{12}A_{21}<0$.
\end{proof}

\begin{theorem}\label{thm:Stable E_3}
The local asymptotic stability of the infectious prey-free equilibrium $E_3$ is established when the conditions $B_{22} < 0$, $B_{11} + B_{33} < 0$, and $B_{11}B_{33} - B_{13}B_{31} > 0$ are satisfied. These conditions ensure that in the absence of infectious prey, the equilibrium state is stable, with both the susceptible prey and predator populations exhibiting stable dynamics.
\end{theorem}

\begin{proof}
The Jacobian matrix evaluated at $E_3$ is given by
\begin{equation}\label{JacE3}
\mathbb{J}_{E_3}=
\begin{bmatrix}
 B_{11}& 
    B_{12}
     & B_{13}\\[1ex] 
  0 & 
  B_{22} & 
    0\\[1ex]
B_{31} & 
   B_{32} & 
   B_{33}
\end{bmatrix}
\end{equation}
where
\begin{align*}
B_{11}&=\frac{b_0(K-2S_3)}{K(1+k_1P_3)} -a_0 -rd_0S_3^{r-1}P_3\\
B_{12}&=-\frac{b_0S_3}{K(1+k_1P_3)}-\frac{e_0S_3}{1+k_2P_3}\\
B_{13}&=\frac{k_1b_0S_3\left(-K+ S_3\right) }{K(1+k_1P_3)^2}-d_0S_3^r \\
B_{22}&=  \frac{e_0S_3}{1+k_2P_3}-a_1 -d_1P_3 \\
B_{31}&=rd_2S_3^{r-1}P_3\\
B_{32}&= d_3P_3\\
B_{33}&= -a_2 +d_2S_3^r
\end{align*}
 Also, $S_3$ and $P_3$  are given in \eqref{cond:s3} and \eqref{cond:P3} respectively. The eigenvalues of the Jacobian matrix at $E_3$ are $\lambda_1=B_{22}$ and the roots of the characteristic polynomial given by
 \begin{equation}\label{charac:s3}
 \lambda^2-(B_{11}+B_{33})\lambda+B_{11}B_{33}-B_{13}B_{31}=0.
 \end{equation}
For the characteristic polynomial in \eqref{charac:s3}, the roots are $\lambda_2+\lambda_3=B_{11}+B_{33}$ and  $\lambda_2\lambda_3=B_{11}B_{33}-B_{13}B_{31}$. Thus, $E_3$ is locally asymptotically stable if $B_{22}<0,~B_{11}+B_{33}<0$ and $B_{11}B_{33}-B_{13}B_{31}>0$.
\end{proof}

\begin{theorem}\label{thm:coexistence}
The local asymptotic stability of the coexistence equilibrium $E_4$ is determined by the conditions $\Psi_1 > 0$, $\Psi_3 > 0$, and $\Psi_1\Psi_2 > \Psi_3$. These conditions ensure that the equilibrium state, where susceptible prey, infectious prey, and predators coexist, is stable. 
\end{theorem}

\begin{proof}
The Jacobian matrix evaluated at $E_4$ is given by
\[
\mathbb{J}_{E_4}=
\begin{bmatrix}
 C_{11}& 
    C_{12}
     & C_{13}\\[1ex] 
  C_{21} & 
  C_{22} & 
    C_{23}\\[1ex]
C_{31} & 
   C_{32} & 
  0
\end{bmatrix}
\]
where
\begin{align*}
 	C_{11} &=\frac{b_0(K-2S^*-I^*)}{K(1 + k_1P^*)}-a_0 -rd_0S^{*r-1}P^* - \frac{e_0I^*}{1+k_2P^*}\\
     C_{12}&= -\frac{b_0S^*}{K(1+k_1P^*)}-\frac{e_0S^*}{1+k_2P^*}\\
     C_{13} &= \frac{k_1b_0S^*\left(-K+ S^*+I^*\right)}{K(1+k_1P^*)^2} -d_0S^{*r}+\frac{k_2e_0S^*I^*}{(1+k_2P^*)^2}\\  
   C_{21} &= \frac{e_0I^*}{1+k_2P^*}\\
    C_{22} &= \frac{e_0S^*}{1+k_2P^*}-a_1 -d_1P^*\\
    C_{23} &=-\frac{k_2e_0S^*I^*}{(1+k_2P^*)^2}-d_1I^*\\
   C_{31} &= rd_2S^{*r-1}P^*\\
    C_{32}&= d_3P^*\\
   C_{33} &= 0
\end{align*}

\noindent The characteristic equation of the Jacobian matrix $\mathbb{J}_{E_4}$ around the coexistence equilibrium point $E_4$ is given by
\begin{equation}
\lambda^3+\Psi_1\lambda^2+\Psi_2\lambda+\Psi_3=0
\end{equation}

where
\begin{align*}
\Psi_1&=-(C_{11}+C_{22})\\
\Psi_2&= C_{11}C_{22}-C_{23}C_{32}-C_{12}C_{21}-C_{13}C_{31}\\
\Psi_3&=C_{11}C_{23}C_{32}+C_{13}C_{22}C_{31}-C_{12}C_{23}C_{31}-C_{13}C_{21}C_{32}
\end{align*}
The positive values of $\Psi_1$ and $\Psi_3$ indicate the stability of the susceptible prey and predator populations, respectively, while the inequality $\Psi_1\Psi_2 > \Psi_3$ ensures the stability of the infectious prey population. Hence, from Routh–Hurwitz criteria, $E_4$ is locally asymptotically stable.
\end{proof}

%%%%%%%%%% Bifurcation
\section{Bifurcation Analysis}\label{sec:bifucation analysis}
%%%%%%%%%%%%%%%%%%%%%%%%%%%%%%%%%%%%%%%%%%%%%%%
\subsection{Co-dimension one bifurcation}
Co-dimension one bifurcations are important occurrences within dynamical systems that arise when a particular parameter reaches a critical threshold. These bifurcations bring about fundamental alterations in the behavior of the system and are linked to the creation, merging, or elimination of equilibrium points or periodic orbits.

\subsubsection{Saddle-node bifurcation}
Saddle-node bifurcation occurs when a pair of equilibrium points, one stable and one unstable, collide and annihilate each other as a parameter crosses a critical threshold. During a saddle-node bifurcation, the stability of the equilibria changes abruptly, resulting in the emergence of new dynamical regimes. This type of bifurcation is significant as it can help explain phenomena such as the onset of instability or the existence of multiple stable states.

\begin{theorem}\label{thm:saddle-node}
The system governed by the model \eqref{Mainsystem} exhibits  a saddle-node bifurcation near the point $E_4$ when the parameter $k_1$ surpasses a critical value denoted as $k_1^*$, given that two conditions are met: the determinant of the Jacobian matrix $\mathbb{J}_{E_4}$ is zero ($\det(\mathbb{J}_{E_4})=0$), and the trace of the Jacobian matrix $\mathbb{J}_{E_4}$ is negative ($\Tr(\mathbb{J}_{E_4})<0$).
\end{theorem}
\begin{proof}
To establish the restriction for the occurrence of saddle-node bifurcation, we use Sotomayor's theorem \cite{Perko13} at the specific value $k_1 = k^*_1$. At this critical point, we observe that $\det(\mathbb{J}_{E_4}) = 0$ and $\Tr(\mathbb{J}_{E_4}) < 0$, indicating the presence of a zero eigenvalue within the Jacobian matrix $\mathbb{J}_{E_4}$. Let us denote the eigenvectors corresponding to this zero eigenvalue as $X$ and $Y$, associated with the matrices $\mathbb{J}_{E_4}$ and $\mathbb{J}^T_{E_4}$, respectively. As a result, we deduce that
                         $$X = (\bar{x}_1, \bar{x}_2, \bar{x}_3)^T$$ 
and
		                 $$Y = (\bar{y}_1, \bar{y}_2, \bar{y}_3)^T,$$ 
where $\bar{x}_1\neq 0$, $\bar{x}_2\not =0$, $\bar{x}_3=1$, $\bar{y}_1\not=0$, $\bar{y}_2\not=0$, and $\bar{y}_3=1$. 

Additionally, let $G = (G_1, G_2, G_3)^T$, where $G_1, G_2$ and $G_3$ are provided in model \eqref{Mainsystem}.
		
		Now
		\begin{align*}
			Y^TG_{k_{1}}(E_4,k^*_1)&=(\bar{y}_1, \bar{y}_2, \bar{y}_3)\Bigg(-\frac{b_0SP}{(1+k_1P)^2}\left(1-\frac{S+I}{K} \right),0,0\Bigg)^T \\
			&=-\frac{b_0\bar{y}_1SP}{(1+k_1P)^2}\left(1-\frac{S+I}{K} \right) \not=0
		\end{align*}
		and
		\begin{equation*}
			Y^T[D^2G(E_4,k^*_1)(X,X)] \not=0 .
		\end{equation*}
Thus, according to Sotomayor's theorem, the model \eqref{Mainsystem} exhibits a saddle-node bifurcation around the point $E_4$ when $k_1$ reaches the critical value $k^*_1$.
\end{proof}

\begin{theorem}\label{thm:saddle-nodek2}
The system governed by the model \eqref{Mainsystem} exhibits a saddle-node bifurcation near the point $E_4$ as the parameter $k_2$ surpasses a critical value denoted as $k_2^*$. This bifurcation occurs on the condition that the determinant of the Jacobian matrix $\mathbb{J}_{E_4}$ is equal to zero ($\det(\mathbb{J}_{E_4})=0$), and the trace of the Jacobian matrix $\mathbb{J}_{E_4}$ is negative ($\Tr(\mathbb{J}_{E_4})<0$).
\end{theorem}
\begin{proof}
The proof can be readily derived from the proof of Theorem \ref{thm:saddle-node} and is therefore omitted.
\end{proof}

\begin{remark}\label{remark: SN-k1-k2}
We observed saddle-node bifurcation in Figure \ref{fig:bif:SNk} as the fear parameters (i.e., $k_1$ and $k_2$) are used as bifurcation parameters in the one-parameter planes.
\end{remark}

\begin{figure}[hbt!]
\begin{center}
\subfigure[]{
    \includegraphics[width=6.15cm, height=6cm]{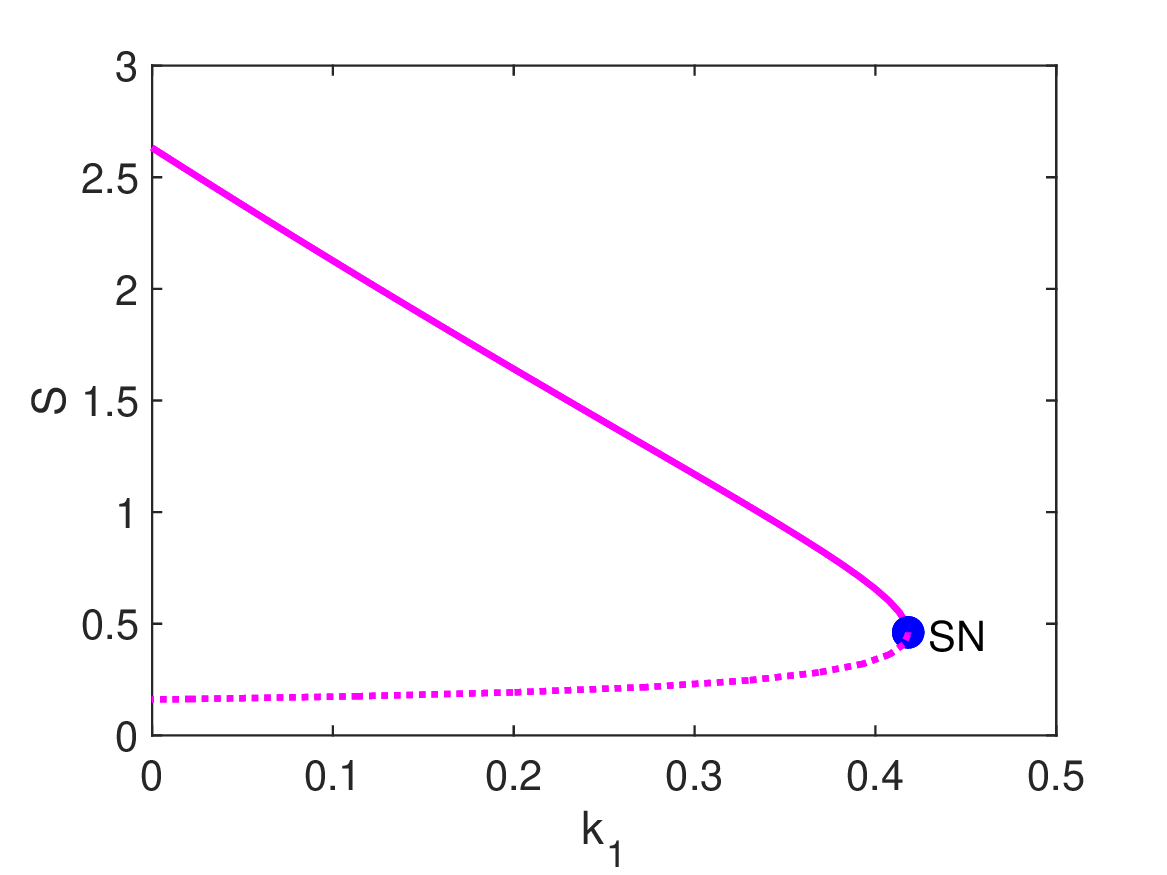}}
\subfigure[]{    
    \includegraphics[width=6.15cm, height=6cm]{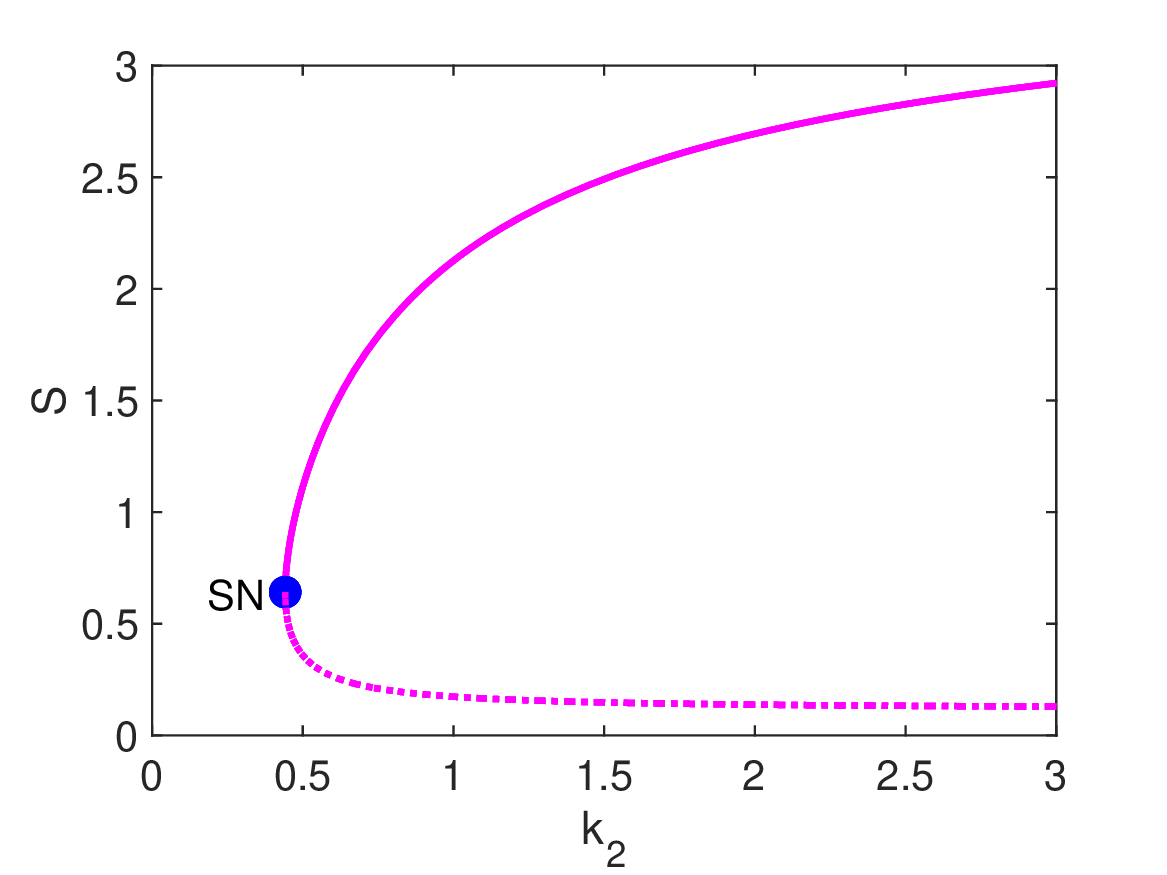}}
\end{center}
\caption{Saddle-node (SN) bifurcation of model \eqref{Mainsystem} as fear parameters are varied (a) $k_2=1$ is fixed and $k_1^*=0.4181$  around $E_4=(0.4615,1.0565,0.8523)$ (b) $k_1=0.1$ is fixed and $k_2^*=0.4417$ around $E_4=(0.6418,0.9591,1.5854)$. All other parameters are fixed and given as $b_0=8,~K=4,~a_0=0.5,~d_0=0.7,~r=0.5,~e_0=4,~a_1=0.4,~d_1=0.7,~a_2=0.8,~d_2=0.4,~d_3=0.5$. Solid line denotes stable and dotted line denotes unstable.}
\label{fig:bif:SNk}
\end{figure}

%%%%%%%%%%%%%%%%%%%%%%%%%%%%%%%%%%%%%%%%%%%%
\subsubsection{Hopf bifurcation}
Hopf bifurcation manifests when the stability of a system undergoes a transition as a parameter crosses a critical threshold. This pivotal point gives rise to the emergence of a limit cycle, resulting in the onset of periodic behavior within the system. This phenomenon is particularly significant as it sheds light on the transformation of system dynamics from equilibrium to sustained oscillations. By studying Hopf bifurcations, researchers gain valuable insights into the intricate interplay between stability and oscillatory behavior, unraveling the underlying mechanisms that govern the appearance and persistence of periodic patterns in diverse systems.

\begin{theorem}\label{thm:hopfk2}
The system governed by the model \eqref{Mainsystem} exhibits a Hopf bifurcation near the point $E_4$ as the parameter $k_2$ surpasses a critical value denoted as $k_2^{H}$. This bifurcation occurs when the following conditions are met:
		\begin{eqnarray}\label{c1}
		\Psi_{1}(k_2^H)>0, ~\Psi_{3}(k_2^H)>0, \quad \Psi_{1}(k_2^H)\Psi_{2}(k_2^H) =\Psi_{3}(k_2^H)
		\end{eqnarray}
		and
		\begin{eqnarray}\label{c2}
		\left[\Psi_{1}(k_2)\Psi_{2}(k_2)\right]^{\prime}_{k_2 = k_2^H} \neq \Psi_{3}^{\prime}(k_2^H).
		\label{eq:}
		\end{eqnarray}
\end{theorem}
\begin{proof}
To observe the occurrence of Hopf bifurcation at $k_2=k_2^{H}$ near the coexistence equilibrium $E_4$, the characteristic equation must exhibit the following form:
		\begin{equation}\label{characeq}
		(\lambda^{2}(k_2^{H}) + \Psi_{2}(k_2^{H}))(\lambda(k_2^{H})+\Psi_{1}(k_2^{H}))=0 ,
		\end{equation}
		which has roots $\lambda_{1}(k_2^{H}) = i \sqrt{\Psi_{2}(k_2^{H})},$  $\lambda_{2}(k_2^{H}) = -i \sqrt{\Psi_{2}(k_2^{H})},$  $\lambda_{3}(k_2^{H}) = - \Psi_{1}(k_2^{H})<0,$ then, $\Psi_3(k_2^{H}) = \Psi_1(k_2^{H})\Psi_2(k_2^{H})$. To establish the presence of Hopf bifurcation at $k_2 =k_2^{H}$ near the coexistence equilibrium point $E_4$, it is necessary to verify the transversality condition:
		\begin{equation}
		\Bigg[\dfrac{\text{d}(Re\lambda_{j}(k_2))}{\text{d}k_2}\Bigg]_{k_2=k_2^{H}}\not=0, j=1,2.
		\end{equation} 
Through the substitution of $\lambda_{j}(k_2) = \eta(k_2)+i\vartheta(k_2)$ into (\ref{characeq}) and subsequent differentiation, we derive the following result:
		
	\begin{eqnarray}
		F_{1}(k_2)\eta^{\prime}(k_2)-F_{2}(k_2)\vartheta^{\prime}(k_2) +F_{4}(k_2) &=& 0, \label{u1}\\
		F_{2}(k_2)\eta^{\prime}(k_2) + F_{1}(k_2)\vartheta^{\prime}(k_2) + F_{3}(k_2)&=&0,\label{u2}
		\end{eqnarray}
		where
		\begin{eqnarray*}
			F_{1}(k_2)&=&3\eta^{2}(k_2)-3\vartheta^{2}(k_2)+\Psi_{2}(k_2)+2\Psi_{1}(k_2)\eta(k_2),\\
			F_{2}(k_2)&=& 6\eta(k_2)\vartheta(k_2)+2\Psi_{1}(k_2)\vartheta(k_2),\\
			F_{3}(k_2)&=&2\eta(k_2)\vartheta(k_2)\Psi_{1}^{\prime}(k_2)+\Psi_{2}^{\prime}(k_2)\vartheta(k_2),\\
			F_{4}(k_2)&=&\Psi_{2}^{\prime}(k_2)\eta(k_2)+\eta^{2}(k_2)\Psi_{1}^{\prime}(k_2)-\vartheta^{2}(k_2)\Psi_{1}^{\prime}(k_2) + \Psi_{3}^{\prime}(k_2).
		\end{eqnarray*}
				
\noindent At $k_2=k_2^{H},$ $\eta(k_2^{H})=0$ and $\vartheta(k_2^{H})=\sqrt{\Psi_{2}(k_2^{H})}$.  We obtain
\begin{eqnarray*}
			F_{1}(k_2^{H}) &=& -2 \Psi_{2}(k_2^{H}),\\
			F_{2}(k_2^{H}) &=& 2 \Psi_{1}(k_2^{H})\sqrt{\Psi_{2}(k_2^{H})},\\
			F_{3}(k_2^{H}) &=& \Psi_{2}^{\prime}(k_2^{H})\sqrt{\Psi_{2}(k_2^{H})},\\
			F_{4}(k_2^{H}) &=& \Psi_{3}^{\prime}(k_2^{H}) - \Psi_{2}(k_2^{H})\Psi_{1}^{\prime}(k_2^{H}).
		\end{eqnarray*}	
Upon solving equations \eqref{u1} and \eqref{u2} for $\eta^{\prime}(k_2^{H})$, we derive the following expression:
		\begin{eqnarray*}
			\left[\frac{\text{d}Re(\lambda_{j}(k_2))}{\text{d}k_2}\right]_{k_2=k_2^{H}}&=&\eta^{\prime}(k_2^{H})\\&=&-\frac{F_{4}(k_2^{H})F_{1}(k_2^{H})+F_{3}(k_2^{H})F_{2}(k_2^{H})}{F_{1}^{2}(k_2^{H})+F_{2}^{2}(k_2^{H})}\\
			&=&\frac{\Psi_{3}^{\prime}(k_2^{H})-\Psi_{2}(k_2^{H})\Psi_{1}^{\prime}(k_2^{H})-\Psi_{1}(k_2^{H})\Psi_{2}^{\prime}(k_2^{H})}{2\left(\Psi_{2}(k_2^{H})+\Psi_{1}^{2}(k_2^{H})\right)}\not=0
		\end{eqnarray*}
				
\noindent on condition that
\begin{align*}
\left[\Psi_{1}(k_2)\Psi_{2}(k_2)\right]^{\prime}_{k_2 = k_2^{H}} \neq \Psi_{3}^{\prime}(k_2^{H}).
\end{align*}

\noindent Consequently, the satisfaction of the transversality condition indicates that the model \eqref{Mainsystem} undergoes a Hopf bifurcation near the coexistence equilibrium point $E_4$ when $k_2=k_2^H$.
\end{proof}

\begin{theorem}\label{thm:hopfk1}
The system governed by the model \eqref{Mainsystem} exhibits a Hopf bifurcation near the point $E_4$ as the parameter $k_1$ surpasses a critical value denoted as $k_1^{H}$. This bifurcation occurs when the following conditions are met:
		\begin{eqnarray}\label{c1}
		\Psi_{1}(k_1^H)>0, ~\Psi_{3}(k_1^H)>0, \quad \Psi_{1}(k_1^H)\Psi_{2}(k_1^H) =\Psi_{3}(k_1^H)
		\end{eqnarray}
		and
		\begin{eqnarray}\label{c2}
		\left[\Psi_{1}(k_1)\Psi_{2}(k_1)\right]^{\prime}_{k_1 = k_1^H} \neq \Psi_{3}^{\prime}(k_1^H).
		\label{eq:}
		\end{eqnarray}
\end{theorem}

\begin{proof}
The proof can be readily derived from the proof of Theorem \ref{thm:hopfk2} and is therefore omitted.
\end{proof}

%%%%%%%%% Transcritical bifurcation
\subsubsection{Transcritical bifurcation}
This bifurcation arises when the stability of two equilibrium points, one stable and one unstable, is interchanged with the variation of a parameter. This stability exchange occurs precisely when the parameter surpasses a critical value. Throughout a transcritical bifurcation, the stable and unstable equilibria coexist both prior to and following the bifurcation point, but their respective roles are reversed.

\begin{theorem}\label{thm:transcriticalk1}
The model \eqref{Mainsystem} experiences transcritical bifurcation around the infectious prey free equilibrium $E_3$ when $k_1$ crosses the critical threshold value $k_1^{TC}$, where $k_1^{TC}$ is computed when $B_{11}|_{k_1=k_1^{TC}}B_{33}-B_{13}|_{k_1=k_1^{TC}}B_{31}=0$. The values of $B_{ij}$ for $i,j=1,2,3$ can be found in equation \eqref{JacE3}.
\end{theorem}

\begin{proof}
One of the eigenvalues of the Jacobian matrix in equation \eqref{JacE3} is given by $B_{22}$, while the other eigenvalues are determined by solving the quadratic equation presented in equation \eqref{charac:s3}. Evaluating the Jacobian matrix of model \eqref{Mainsystem} at equilibrium point $E_3$, we find that it has a zero eigenvalue if $B_{11}B_{33}-B_{13}B_{31}=0$, which implies $k_1=k_1^{TC}$.

Next, we determine the eigenvectors $M_1$ and $M_2$ associated with the zero eigenvalue of the matrices $J_{E_3}^{TC}$ and $(J_{E_3}^{TC})^T$, respectively. We find that $M_1=(\sigma_1,\sigma_2,\sigma_3)^T$ and $M_2=(\varrho_1,\varrho_2,\varrho_3)^T$. Furthermore, we observe the following conditions: \\
$M_{2}^{T}G_{k_1}\left(E_3,k_1^{TC} \right)=0$, $M_{2}^{T}\left[DG_{k_1}\left(E_3,k_1^{TC} \right)M_1\right] \neq 0$, and\\
$M_{2}^{T}\left[D^2 G\left(E_3,k_1^{TC} \right)(M_1,M_1)\right]\neq 0$. Based on these conditions, we conclude that model \eqref{Mainsystem} undergoes a transcritical bifurcation at equilibrium point $E_3$ when $k_1=k_1^{TC}$, as the transversality conditions are satisfied according to Sotomayor's theorem \cite{Perko13}. 

\end{proof}

\begin{theorem}\label{thm:transcriticalk2}
The model \eqref{Mainsystem} experiences transcritical bifurcation around the infectious prey free equilibrium $E_3$ when $k_2$ crosses the critical threshold value $k_2^{TC}$, where $k_2^{TC}$ is computed when $B_{11}|_{k_2=k_2^{TC}}B_{33}-B_{13}|_{k_2=k_2^{TC}}B_{31}=0$. The values of $B_{ij}$ for $i,j=1,2,3$ can be found in equation \eqref{JacE3}.
\end{theorem}
\begin{proof}
The proof can be readily derived from the proof of Theorem \ref{thm:transcriticalk1} and is therefore omitted.
\end{proof}

Next, we present numerical simulations to corroborate theorems \ref{thm:hopfk2}, \ref{thm:hopfk1}, \ref{thm:transcriticalk1} and \ref{thm:transcriticalk2} in Figure \ref{fig:HTCk1k2}.

%%%%%%%% HTC
\begin{figure}[hbt!]
\begin{center}
   \subfigure[]{
   \includegraphics[width=6.15cm, height=6.1cm]{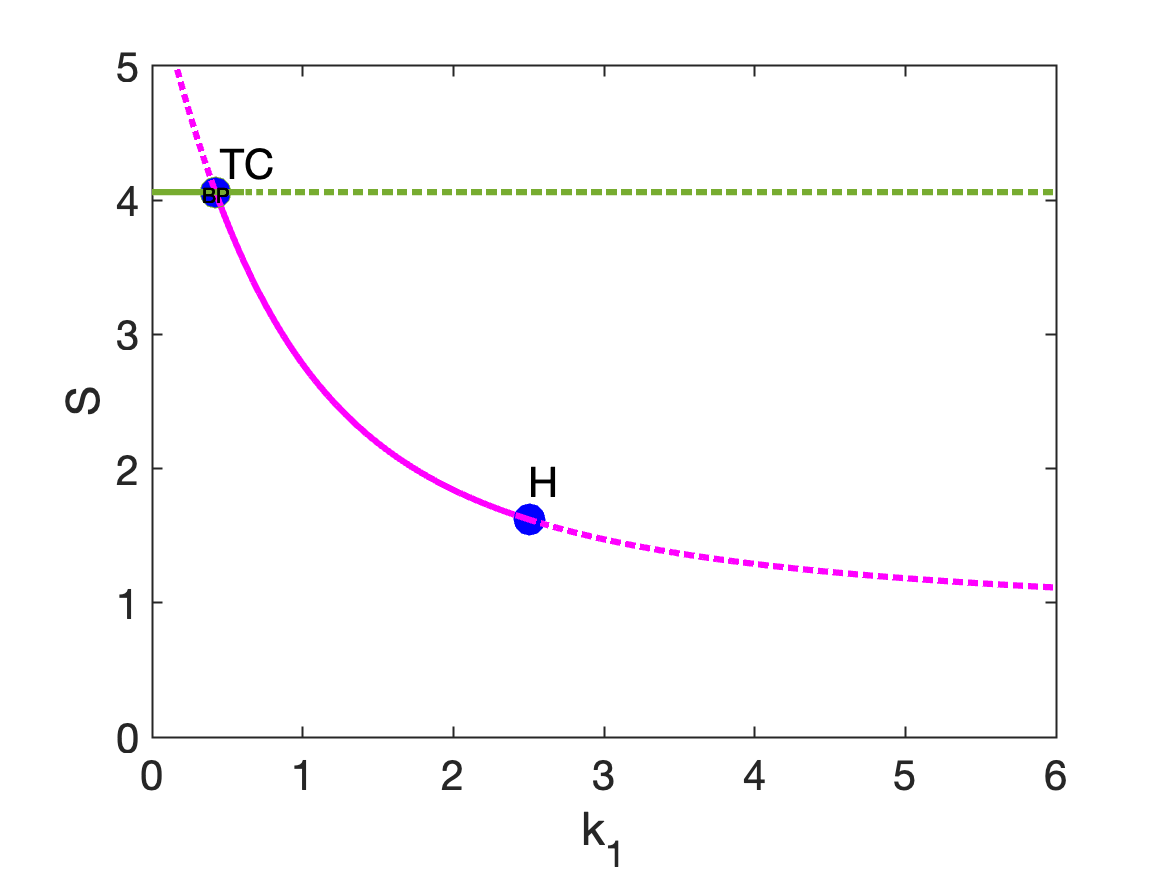}}
\subfigure[]{    
    \includegraphics[width=6.15cm, height=6.1cm]{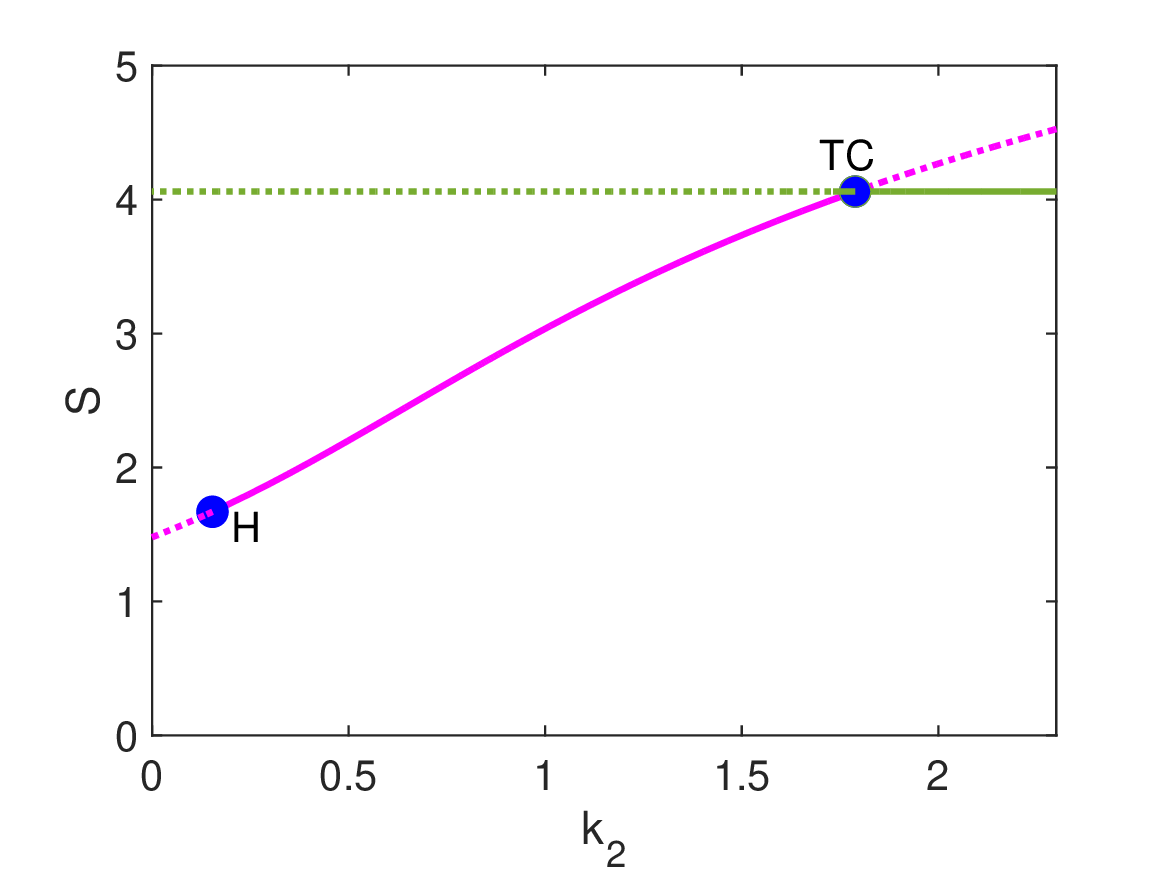}}
\end{center}
\caption{ Hopf and transcritical bifurcations of model \eqref{Mainsystem} as the fear parameters are varied. (a)  Transcritical point (TC) at $k_1^{TC}=0.4219$ around $E_3=(4.0600,0,0.9978)$. Also, the Hopf point (H) $k_1^H=2.5075$ around $E_4=(1.6184,0.7596,0.3308)$ with Lyapunov coefficient of $L_{k_1}=-8.827\times 10^{-3}$, (b) Hopf point $k_2^H=0.1536$ around $E_4=(1.6694,0.7411,0.5311)$ with Lyapunov coefficient of $L_{k_2}=-9.9094\times 10^{-3}$. Also, the transcritical point at $k_2^{TC}=1.7885$ around $E_3=(4.0600,0,0.7081)$. All other parameters are fixed and given as $b_0=2,k_1=0.99,~k_2=0.85,~K=8,~a_0=0.3,~d_0=0.6,~r=0.7,~e_0=0.5,~a_1=0.4,~d_1=0.7,~a_2=0.8,~d_2=0.3,~d_3=0.5$. Solid line denotes stable and dotted line denotes unstable.}
\label{fig:HTCk1k2}
\end{figure}

\subsection{Co-dimension two bifurcation}
Additionally, our investigation delves into the potential existence of a diverse set of co-dimension two bifurcations within the model \eqref{Mainsystem}. These complex bifurcations involve the simultaneous variation of two key parameters and can lead to intricate dynamic behaviors not observed in simpler models.

\subsubsection{Zero-Hopf bifurcation}
The Zero-Hopf bifurcation is a type of unfolding that occurs in a 3-dimensional autonomous differential system, involving two parameters, and characterized by a zero-Hopf equilibrium \cite{LM16, L14}. 

\begin{definition}[Zero-Hopf Equilibrium]
A zero-Hopf equilibrium is a specific type of equilibrium point found in a 3-dimensional autonomous differential system. It is characterized by having a simple zero eigenvalue and a simple pair of purely imaginary eigenvalues.
\end{definition}

By continuation of the saddle-node point or limit point (i.e. $k_2=0.4417$ at ($(0.6418,0.9591,1.5854)$)) with $(k_2,d_0)$ as free parameters, MATCONT package \cite{G05} in MATLAB R2023a detects  a zero-Hopf point ($ZH$) for $k_2=0.9917$ at  $(0.5120,1.0275,0.9416)$ and $d_0=1.4225$, and then numerically compute the eigenvalues. The eigenvalues are given as $\lambda=0$ and $\lambda=\pm 2.2402i$. This two-parameter bifurcation is illustrated in Figures \ref{fig:ZH-SNTC}(a) and (b). Figure \ref{fig:ZH-SNTC}(a) is the 3D representation of the parameters $(k_2,d_0)$ versus the susceptible prey population $S$. These are unfolded in Figure \ref{fig:ZH-SNTC}(b) on the two-parameter curve $(k_2,d_0)$. Again in Figure \ref{fig:ZH-SNTC}(c), we observed two zero-Hopf points when $(k_2, K)$ are used as free parameters. Herein, $ZH_1$ is the first zero-Hopf point and occurred for $(k_2,K)=(0.2868,5.0261)$ at $(0.5394,1.0125,1.5586)$. The eigenvalues associated with $ZH_1$ are $\lambda=0$ and $\lambda=\pm 2.6876i$.  Furthermore,  $ZH_2$ is the second zero-Hopf point and occurred for $(k_2,K)=(0.2585,5.5590)$ at $(0.1487,1.2915,0.2304)$. The eigenvalues associated with $ZH_2$ are $\lambda=0$ and $\lambda=\pm 1.9754i$.

Next, we formulate conjectures about the possible existence of zero-Hopf bifurcation to summarize the numerical findings above.

\begin{conjecture}\label{conj:ZH k_2,b_0}
Consider the model \eqref{Mainsystem} with all parameters fixed except $(k_2,d_0)$.  Then, there exist values  $(k_2^*,d_0^*)$ for which the model \eqref{Mainsystem} undergoes a zero-Hopf bifurcation.
\end{conjecture}

\begin{conjecture}\label{conj:ZH k_2, K}
Consider the model \eqref{Mainsystem} with all parameters fixed except $(k_2,K)$.  Then, there exist critical points $(k_2^*,K^*)$ such that the model \eqref{Mainsystem} undergoes multiple zero-Hopf bifurcations.
\end{conjecture}

\subsubsection{Saddle-node transcritical bifurcation}
The saddle-node-transcritical bifurcation ($SNTC$) arises from the intersection of curves corresponding to the saddle-node and transcritical bifurcations. This creates an intriguing scenario where the stability and existence of equilibrium points undergo a sudden change, impacting the system's behavior significantly. For more recent works regarding $SNTC$, please see \cite{SV10, PWAB23} and references therein.

By continuation of the saddle-node point or limit point (i.e. $k_2=0.4417$ at $(0.6418,0.9591,1.5854)$) with $(k_2,K)$ as free parameters, MATCONT package in MATLAB version R2023a detects multiple saddle-node transcritical points $SNTC_1$ and $SNTC_2$. $SNTC_1$ occurred for $k_2=0.4508$ at  $(0.6429,0.9586,1.5812)$ and $K=3.9784$. $SNTC_2$ occurred for $k_2=0.2271$ at  $(0.2602,1.1732,0.8042)$ and $K=5.7454$. These $SNTC's$ are classified into the elliptic case i.e. single zero eigenvalues. This two-parameter bifurcation is illustrated in Figure \ref{fig:ZH-SNTC}(c) and (d). We present a 3D bifurcation surface to this effect depicting the parameters $(k_2,K)$ and the susceptible prey state in Figure \ref{fig:ZH-SNTC}(c). These are unfolded in Figure \ref{fig:ZH-SNTC}(d) on the two-parameter curve $(k_2,K)$.

Next, we formulate a conjecture about saddle-node transcritical bifurcation to summarize the numerical results above.

\begin{conjecture}\label{conj:SNTC k_2,K}
Consider the model \eqref{Mainsystem} with all parameters fixed except $(k_2,K)$.  Then, there exists critical points for which the model \eqref{Mainsystem} experiences multiple saddle-node transcritical bifurcations.
\end{conjecture}

\begin{remark}
It is interesting to observed the loop formed in Figure \ref{fig:ZH-SNTC}(c) as the parameters $(k_2,K)$ are varied. This loop contains two saddle-node transcritical and zero-Hopf points. As far as we are aware, this is the inaugural work in the field of eco-epidemiology incorporating fear effect to demonstrate the occurrence of saddle-node transcritical  and zero-Hopf bifurcations.
\end{remark}

%%%%%%%% ZH and SNTC
\begin{figure}[hbt!]
\begin{center}
\subfigure[]{
    \includegraphics[width=6.15cm, height=6.1cm]{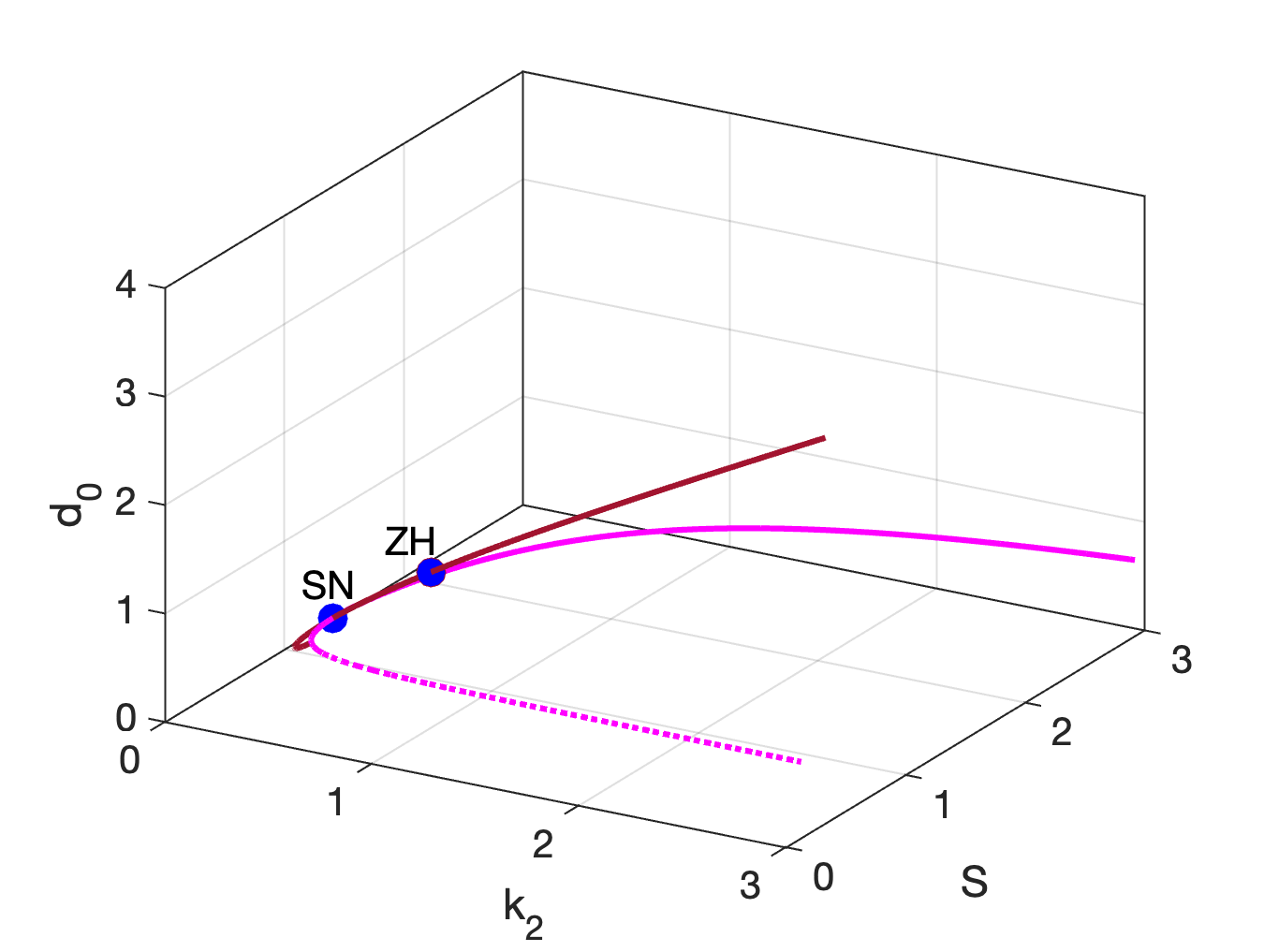}}
\subfigure[]{
    \includegraphics[width=6.15cm, height=6.1cm]{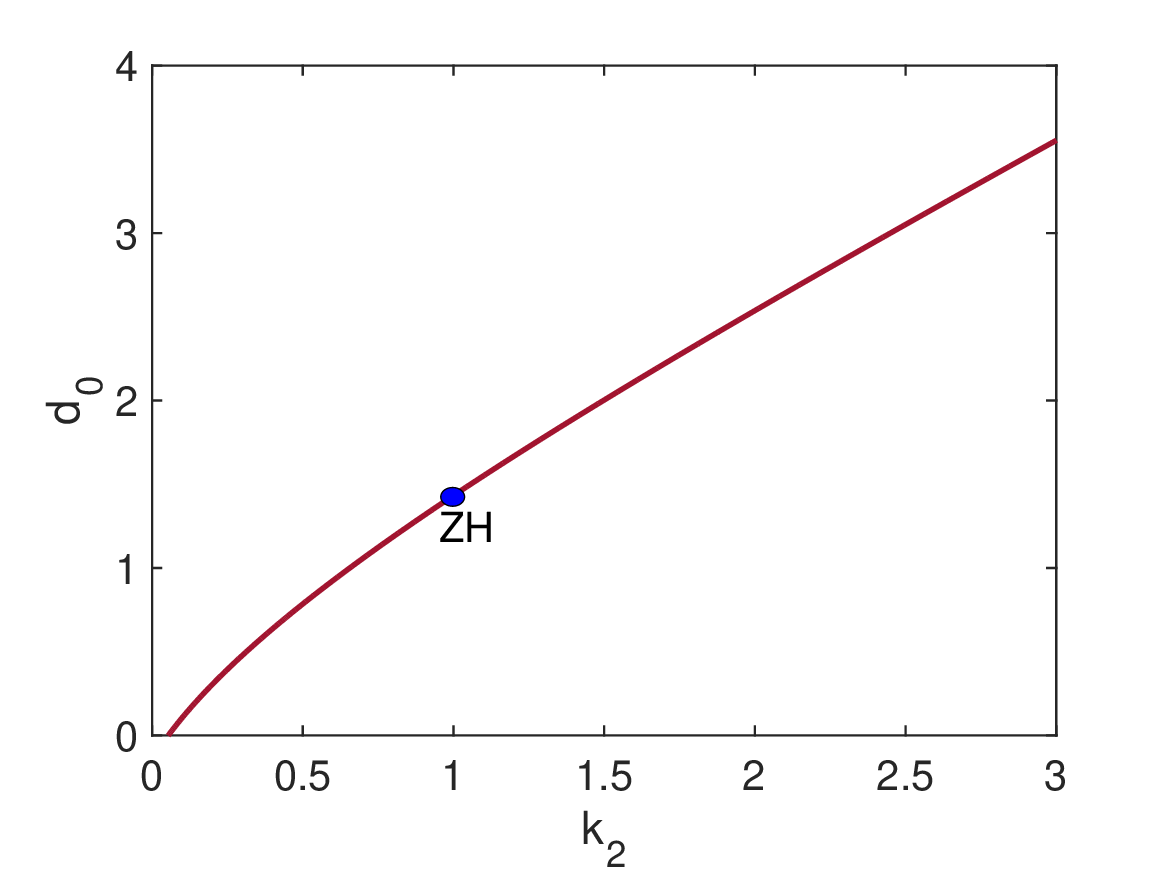}}
 \subfigure[]{    
    \includegraphics[width=6.15cm, height=6.1cm]{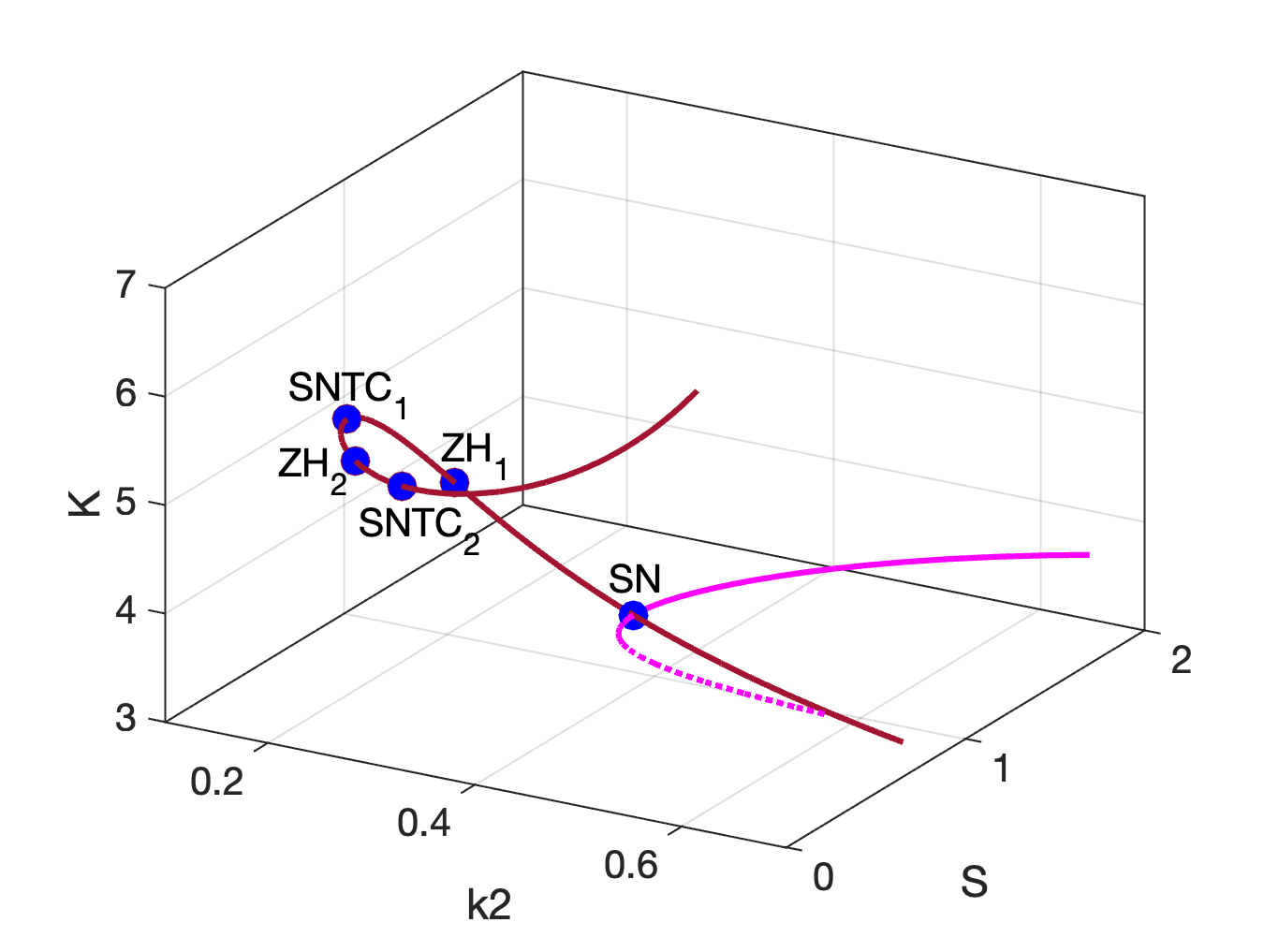}} 
\subfigure[]{    
    \includegraphics[width=6.15cm, height=6.1cm]{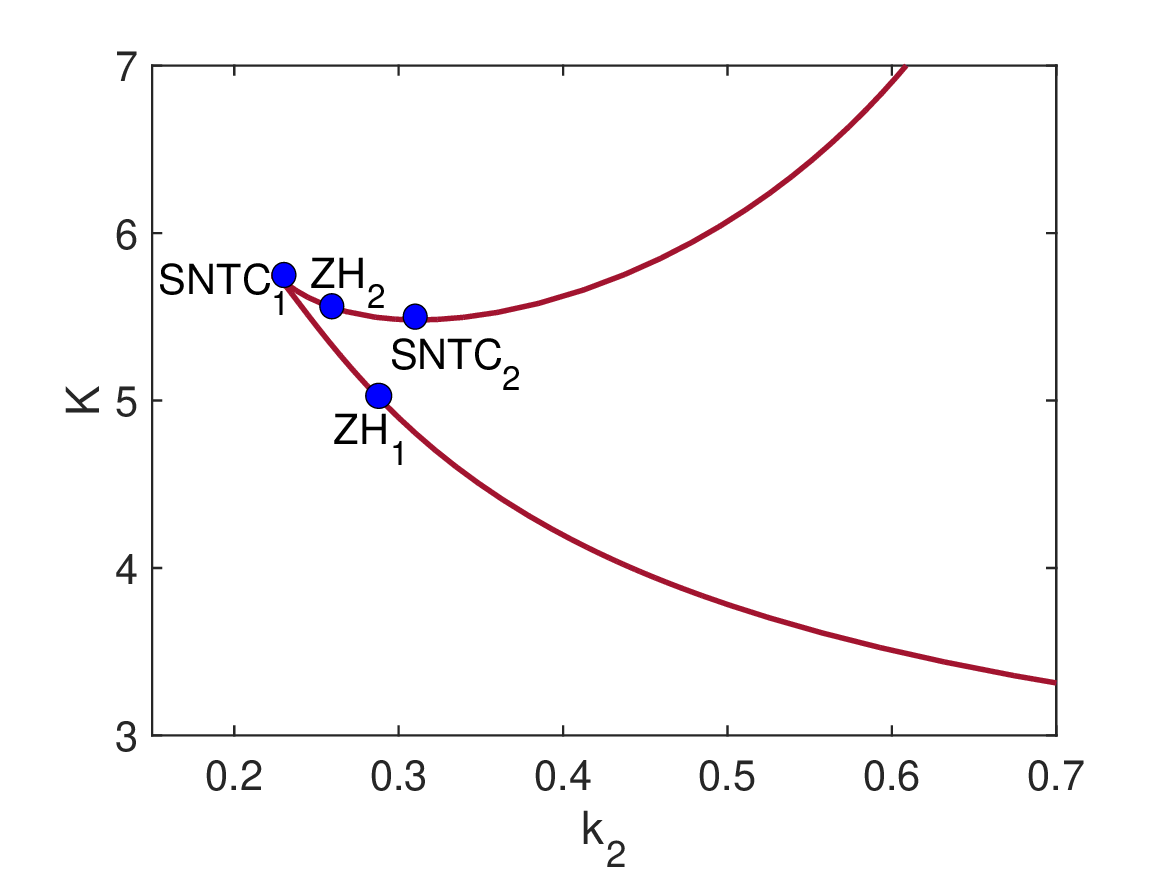}}     
\end{center}
\caption{Two-parameter bifurcation diagrams for model \eqref{Mainsystem} as key parameters are varied. (a) 3-D bifurcation surface of the saddle-node and zero-Hopf curves. (b) Zero-Hopf ($ZH$) point at $(k_2,d_0)=(0.9917,1.4225)$. (c) 3-D bifurcation surface of the saddle-node, saddle-node transcritical, and zero-Hopf curves.   (d) Multiple  zero-Hopf  points at $(k_2,K)=(0.2868,5.0261)$ and $(k_2,K)=(0.2585,5.5590)$. Multiple saddle-node transcritical ($SNTC$) points at $(k_2,K)=(0.4508,3.9784)$ and $(k_2,K)=(0.2271,5.7454)$. All other parameters are fixed and given as $b_0=8,~K=4,~a_0=0.5,~d_0=0.7,~r=0.5,~e_0=4,~a_1=0.4,~d_1=0.7,~a_2=0.8,~k_1=0.1,~d_2=0.4,~d_3=0.5$.}
\label{fig:ZH-SNTC}
\end{figure}

%%%%%%%%%%%%%%%%%%%%%%%%%%%%%%%%%%%%%%%%%% FTE
\section{Finite Time Extinction of Susceptible Prey}\label{sec:FTE}
\noindent In this section, we shall investigate the possibility of the susceptible population going extinct in finite time when a key parameter is varied.

\begin{definition}[Finite Time Extinction]
Finite time extinction in population dynamics refers to the scenario where a population completely vanishes within a finite period. It implies the absence of any individuals in the population after a certain time, leading to its extinction. Mathematically, finite time extinction of a population (say the susceptible population) can be defined as follows:
\[\lim_{t\rightarrow t^*}S(t)=0\]
where $t^*$ is a finite time point at which the population becomes extinct. 
\end{definition}

\begin{remark}\label{remart: FTE}
In Figure \ref{fig:FTE}, we observe the time evolution of population densities when the fear of predators that reduces the birth rate of susceptible prey parameter $k_1$ is varied. A stable dynamics of the endemic state is seen in Figure \ref{fig:FTE}(a) when there is no fear of predators, i.e. $k_1=0$. At $k_1=0.2$, we observe an extinction of the susceptible prey in finite time leading to the eventual extinction of both infected prey and predator populations, see Figure \ref{fig:FTE}(b). Additionally, as depicted in Figure \ref{fig:FTE}(b), we note that the specific finite time point at which the susceptible prey population reaches extinction corresponds to $t^*\approx 14.3$. This temporal value is corroborated via Mathematica version 13.2.1.
\end{remark}

%%%%%%%% FTE
\begin{figure}[hbt!]
\begin{center}
   \subfigure[]{
    \includegraphics[width=6.15cm, height=6.1cm]{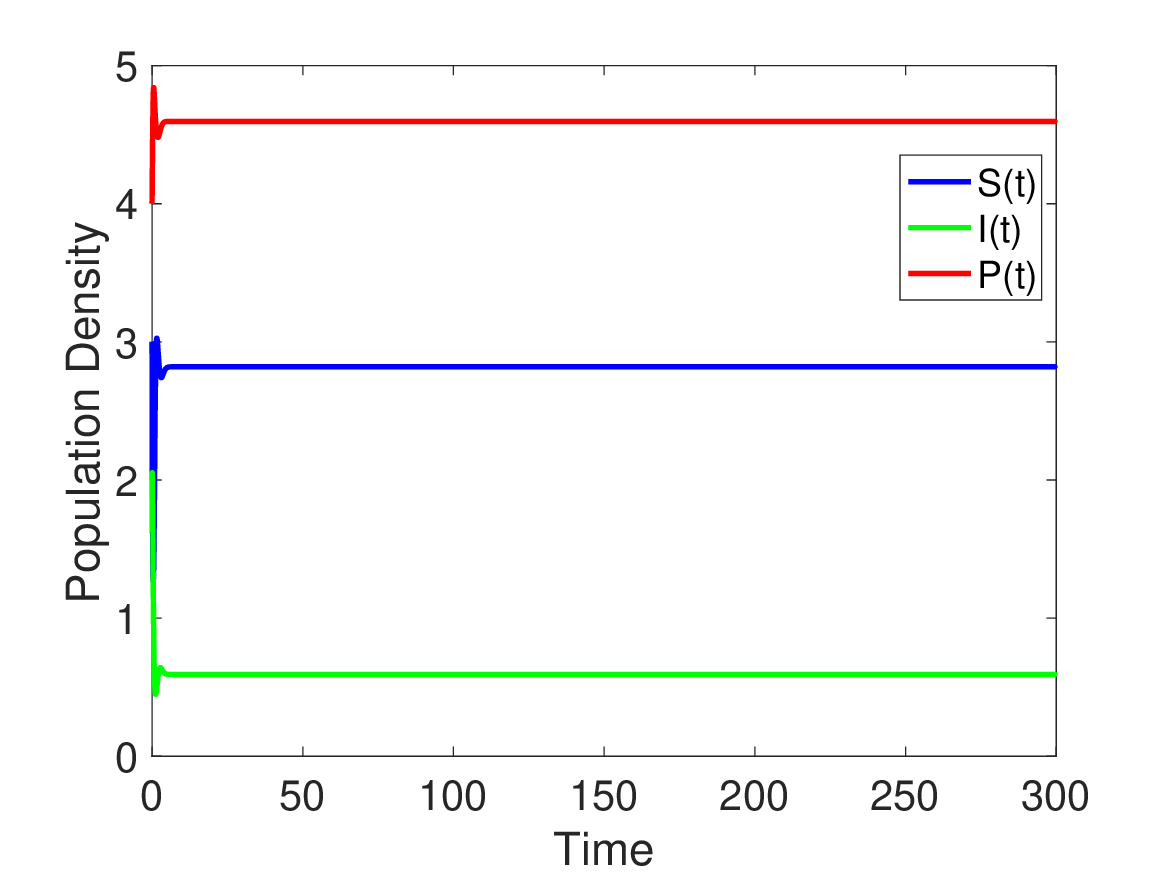}}
\subfigure[]{    
    \includegraphics[width=6.15cm, height=6.1cm]{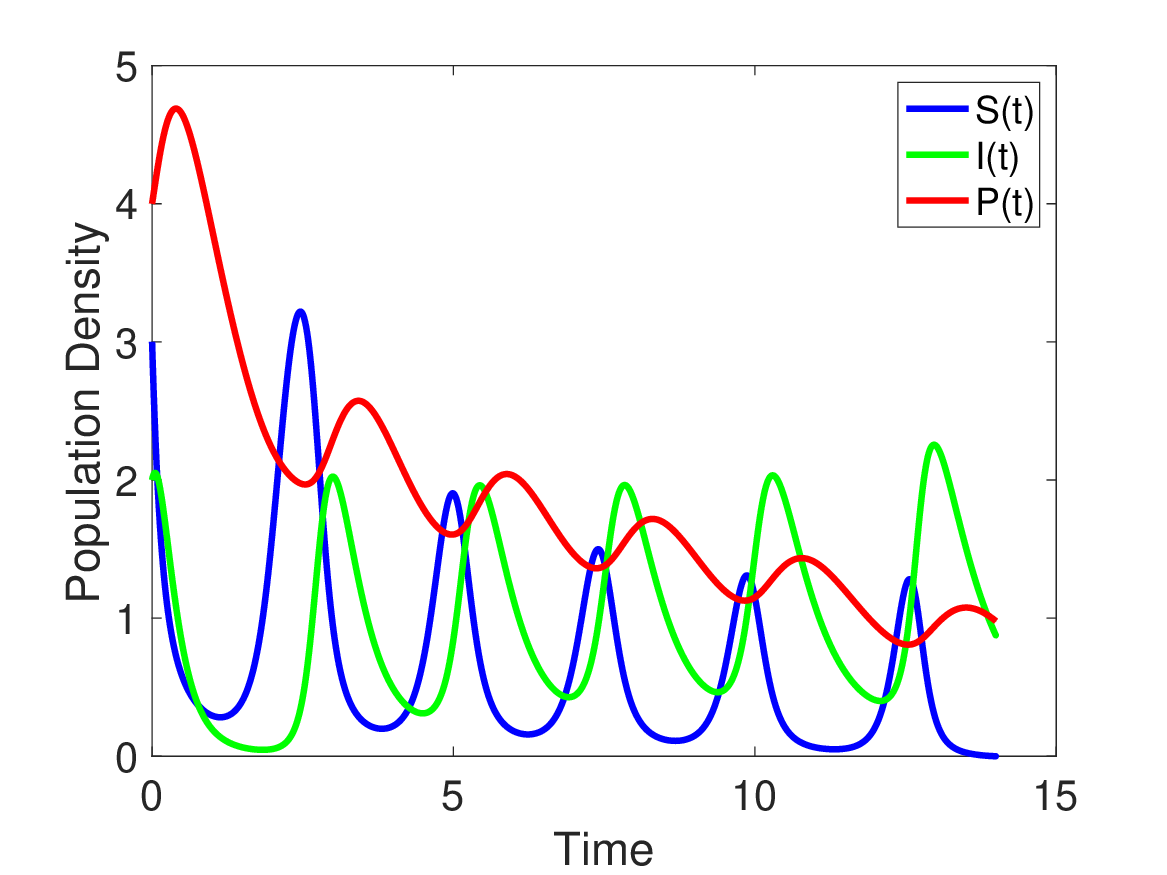}}
\end{center}
\caption{Time evolution of population densities with varying $k_1$. (a) Stable dynamics observed at $k_1=0$ for $E_4=(2.8194,0.5925,4.5959)$ (b) Finte time extinction of $S$ when $k_1=0.2$. All other parameters are fixed and given as $b_0=10,~K=5,~a_0=0.5,~d_0=0.7,~r=0.5,~e_0=6,~a_1=0.4,~d_1=0.7,~a_2=0.8,~k_2=0.8,~d_2=0.3,~d_3=0.5$ and initial condition is given as $(3,2,4)$.}
\label{fig:FTE}
\end{figure}

\section{Disease Management Strategies}\label{sec:Disease management}
\noindent The objective of disease management strategies is twofold: 
\begin{itemize}
\item to regulate the transmission of disease within the prey population and
\item uphold the stability of predator-prey dynamics.
\end{itemize} 
 Next, we will discuss disease management strategies via fear-based mediated disease control.
\subsection{Fear-based mediated disease control} 
Fear-based mediated disease control refers to a disease management approach that utilizes fear as a mechanism to control the spread of infectious diseases. In this strategy, the fear of predation is leveraged to induce behavioral changes in susceptible preys, reducing their contact and interactions with infected prey and thereby minimizing disease transmission.

In the context of the proposed eco-epidemiological model \eqref{Mainsystem}, the dynamics of the system reveal interesting phenomena regarding the interactions between the susceptible prey population, infected prey population, and predators. Specifically, the infected prey population diminishes as the fear of predators that suppresses the birth rate of susceptible prey, represented by the parameter $k_1$, decreases, see Figure \ref{fig:bif:hopfTCk1}. Thus, a decrease in fear results in a reduced tendency of the prey individuals to avoid or flee from predators. As a consequence, the predators are more successful in capturing and consuming the infected prey, thereby reducing the spread of the disease within the population. 

Additionally, as the fear parameter that reduces disease infection, denoted as $k_2$, increases, it leads to the decline and eventual extinction of the infected prey population, see Figure \ref{fig:bif:hopfTCk2}. This can be attributed to the heightened aversion or precautions taken by the susceptible prey population whiles interacting with the infectious prey population, reducing their exposure to the disease and limiting its spread within the population.  

Both of these factors contribute to the eradication of the disease from the ecosystem. The combination of increased fear towards disease transmission and decreased fear towards predators acts as a mechanism that effectively controls and eliminates the disease, leading to the extinction of the infected prey population. This highlights the significant role that fear-based factors play in shaping the dynamics and disease within an ecosystem.

%%%%%%% H & TC k1
\begin{figure}[hbt!]
\begin{center}
   \subfigure[]{
    \includegraphics[width=4.02cm, height=4.02cm]{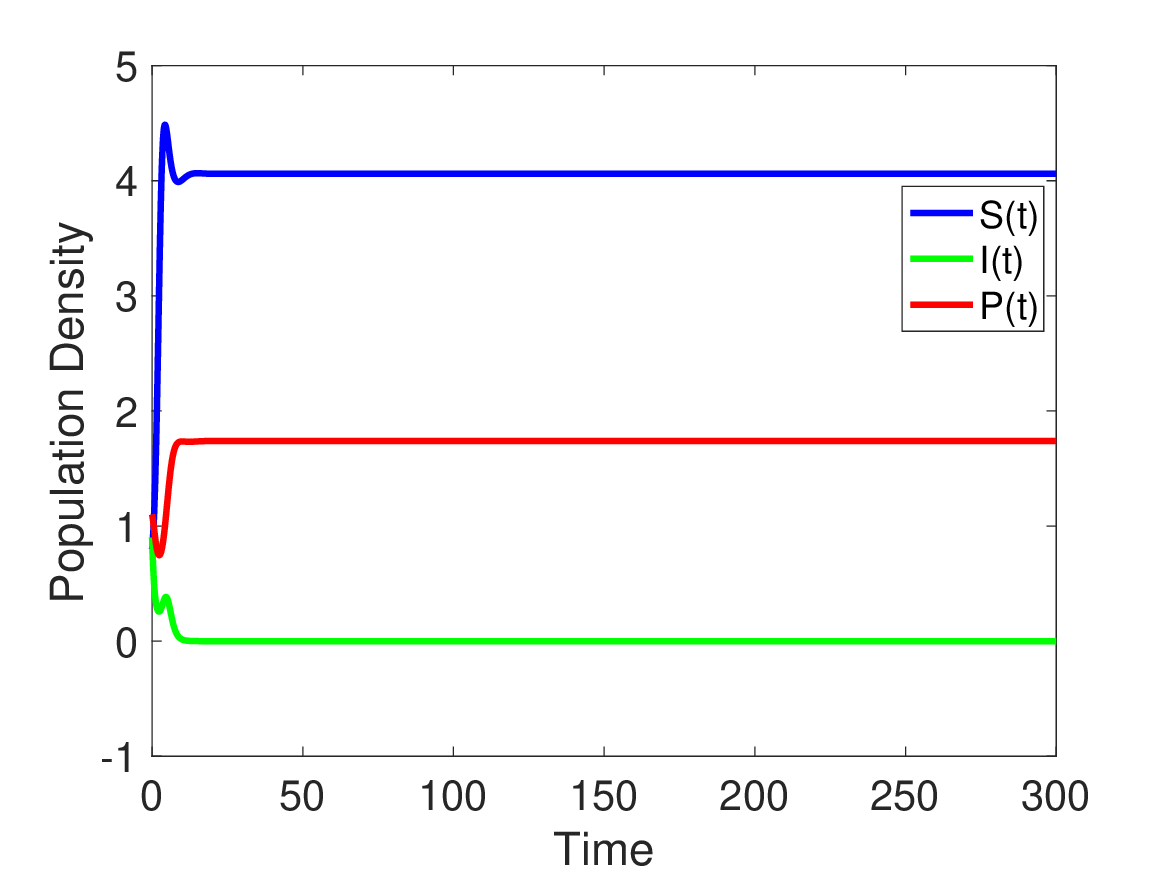}}
\subfigure[]{    
    \includegraphics[width=4.02cm, height=4.02cm]{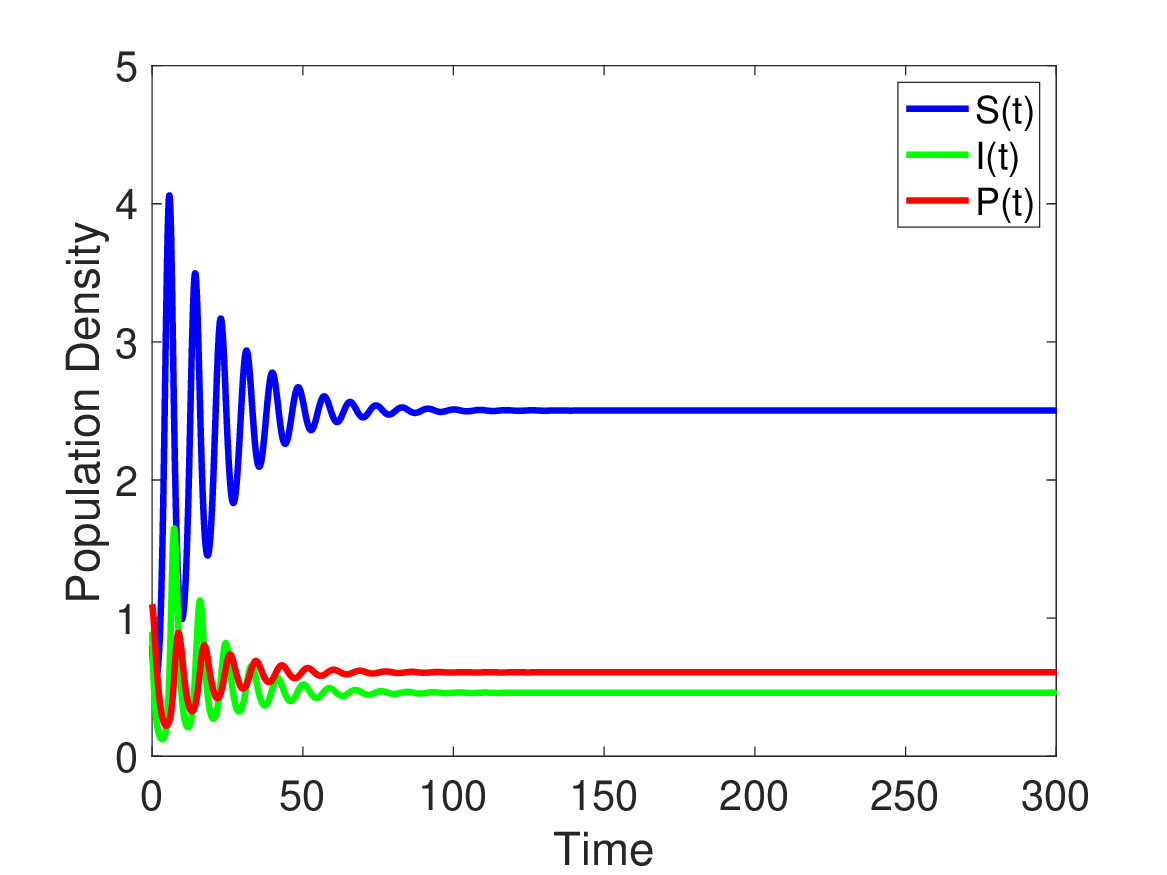}}
    \subfigure[]{
    \includegraphics[width=4.02cm, height=4.02cm]{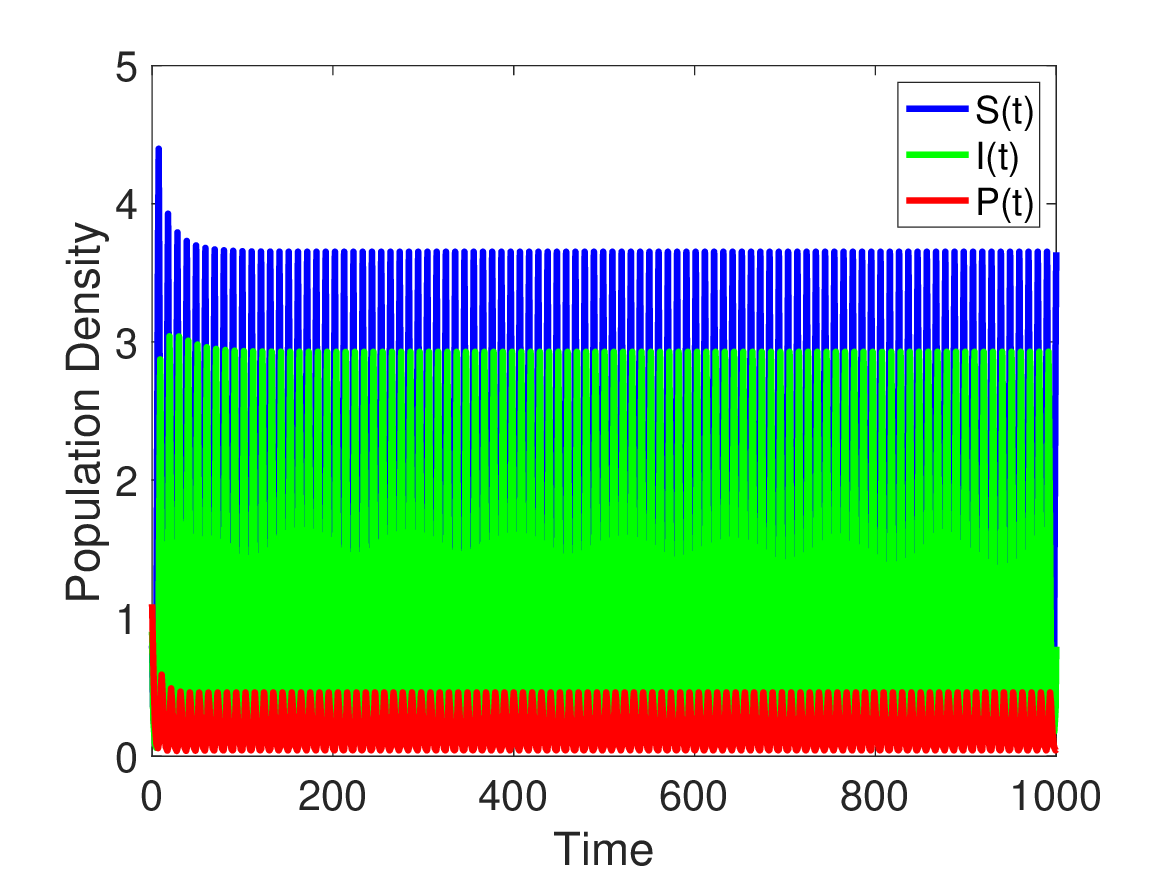}}
%%%%%%% %%%   
    \subfigure[]{
    \includegraphics[width=4.02cm, height=4.02cm]{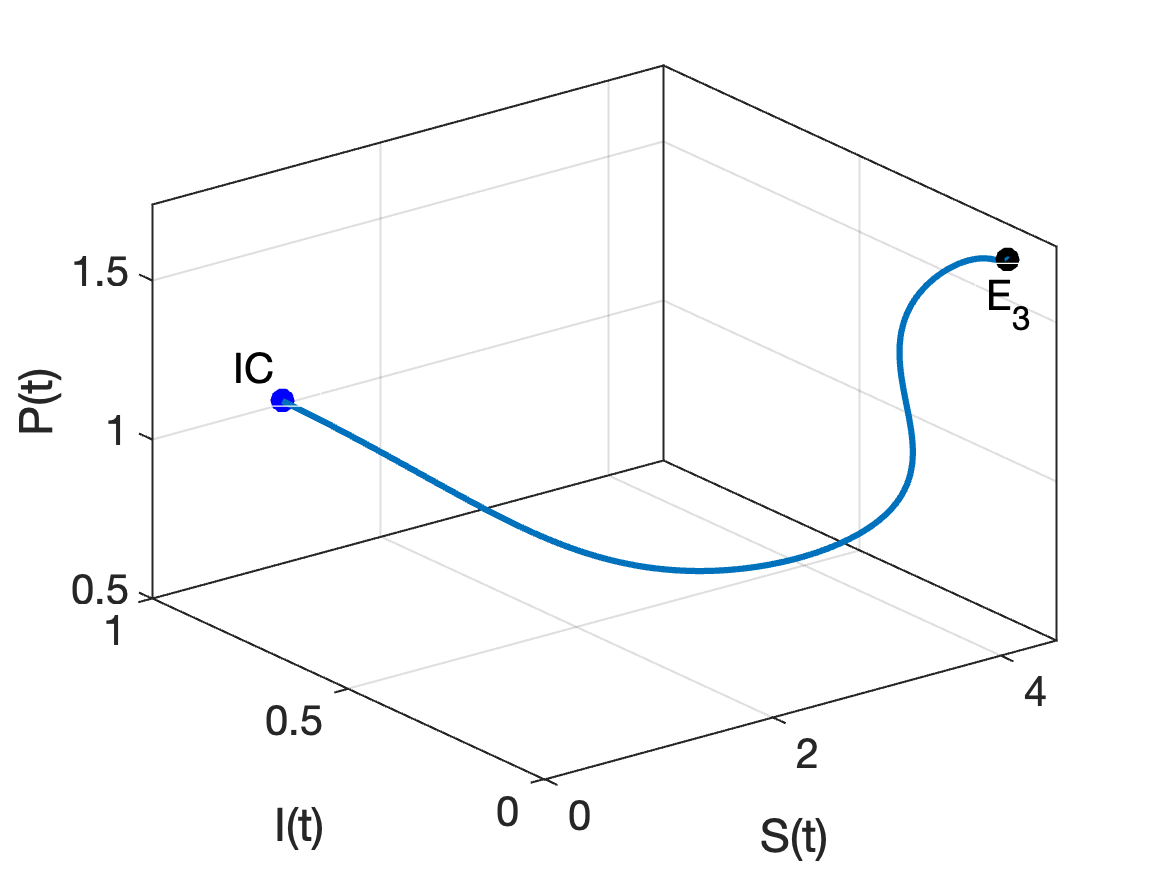}}
\subfigure[]{    
    \includegraphics[width=4.02cm, height=4.02cm]{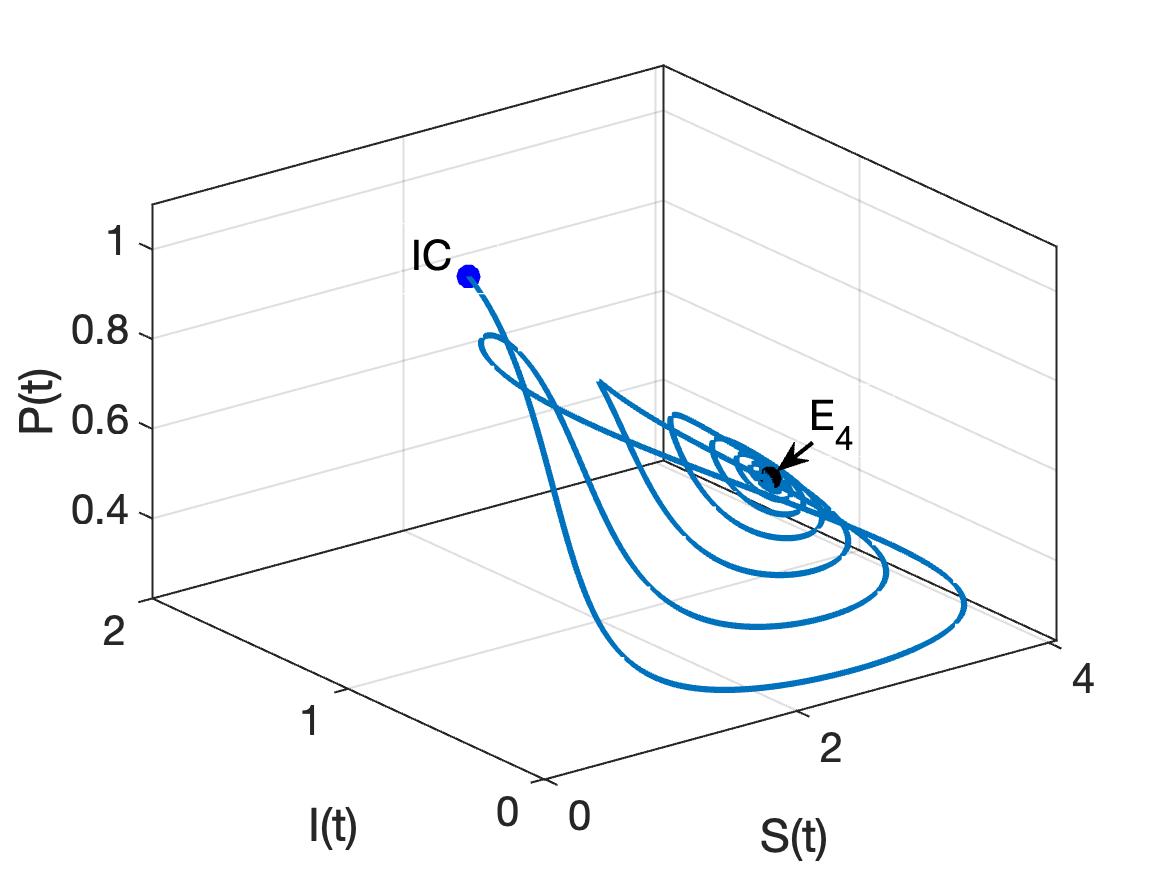}}
    \subfigure[]{
    \includegraphics[width=4.02cm, height=4.02cm]{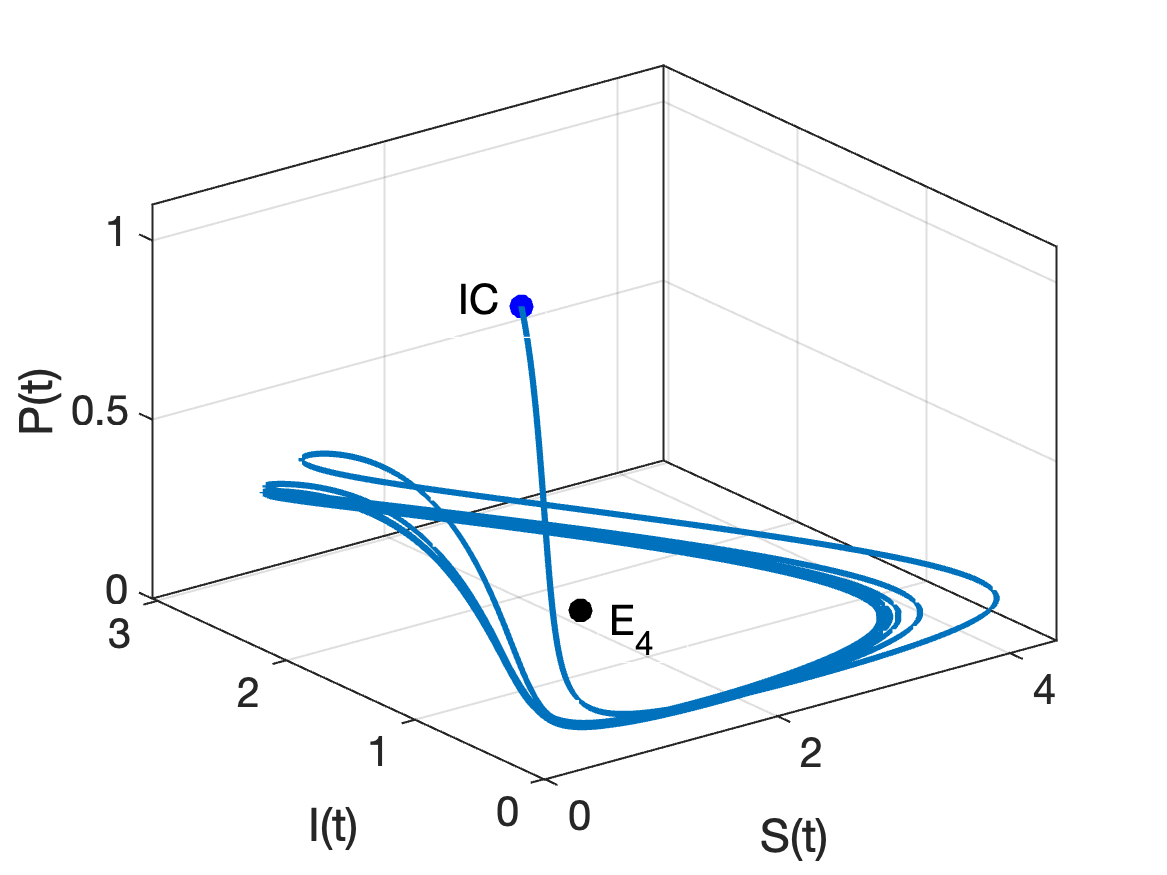}}
%%%%%%%%%%    
%\subfigure[]{    
%    \includegraphics[width=4.02cm, height=4.02cm]{bifHTCk1.eps}}
%    \subfigure[]{    
%    \includegraphics[width=4cm, height=4cm]{bifHTCk1.eps}}
\end{center}
\caption{Figures depicting time evolution of population densities and phase portrait with IC$=(0.8,0.9,1.1)$. ((a) \& (d)) Stable $E_3=(4.0600,0,1.7382)$  at $k_1=0$, ((b) \& (e)) stable $E_4=(2.5030,0.4596,0.6076)$ at $k_1=1.2$, ((c) \& (f)) oscillatory coexistence at $k_1=4$ where $E_4=(1.2898,0.8830,0.2102)$. All other parameters are fixed and given as $b_0=2,~K=8,~a_0=0.3,~d_0=0.6,~r=0.7,~e_0=0.5,~k_2=0.85,~a_1=0.4,~d_1=0.7,~a_2=0.8,~d_2=0.3,~d_3=0.5$.}
\label{fig:bif:hopfTCk1}
\end{figure}

%%%%%%%% H & TC k2
\begin{figure}[hbt!]
\begin{center}
   \subfigure[]{
    \includegraphics[width=4cm, height=4cm]{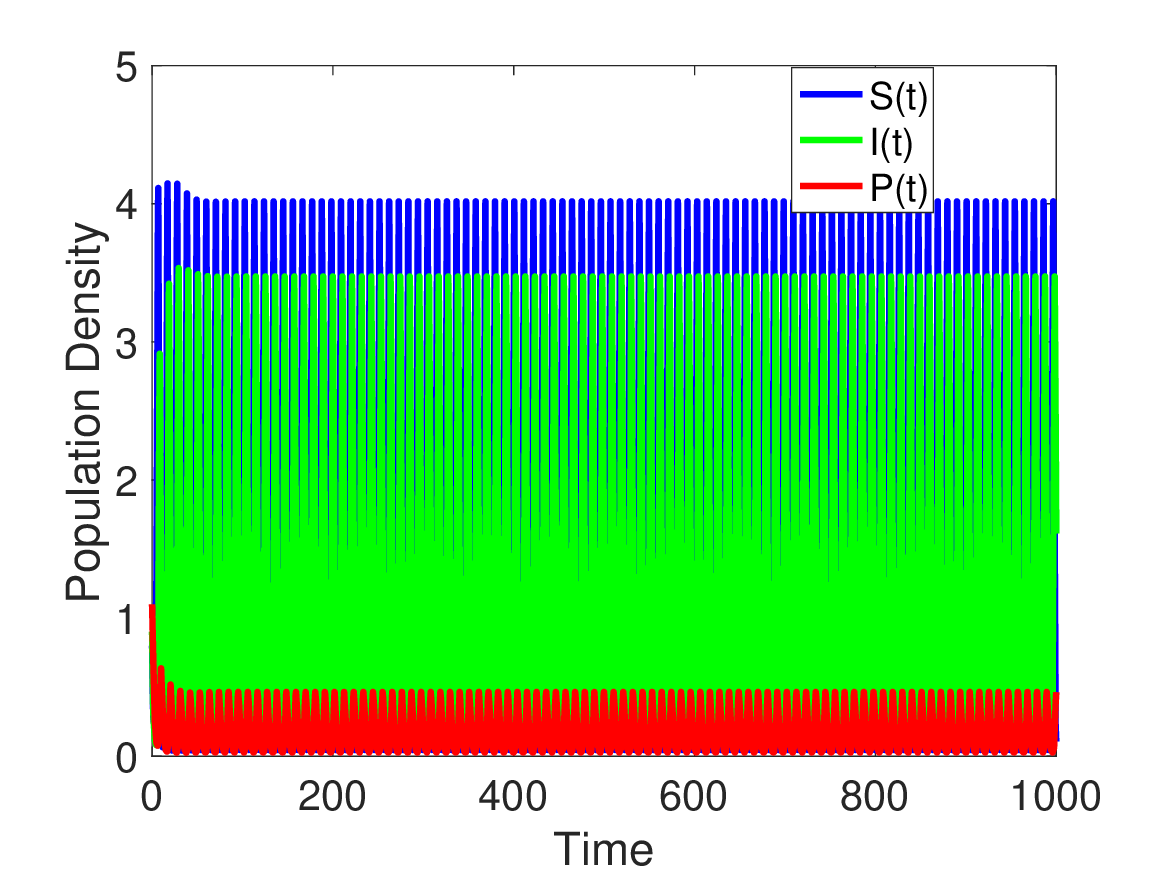}}
\subfigure[]{    
    \includegraphics[width=4cm, height=4cm]{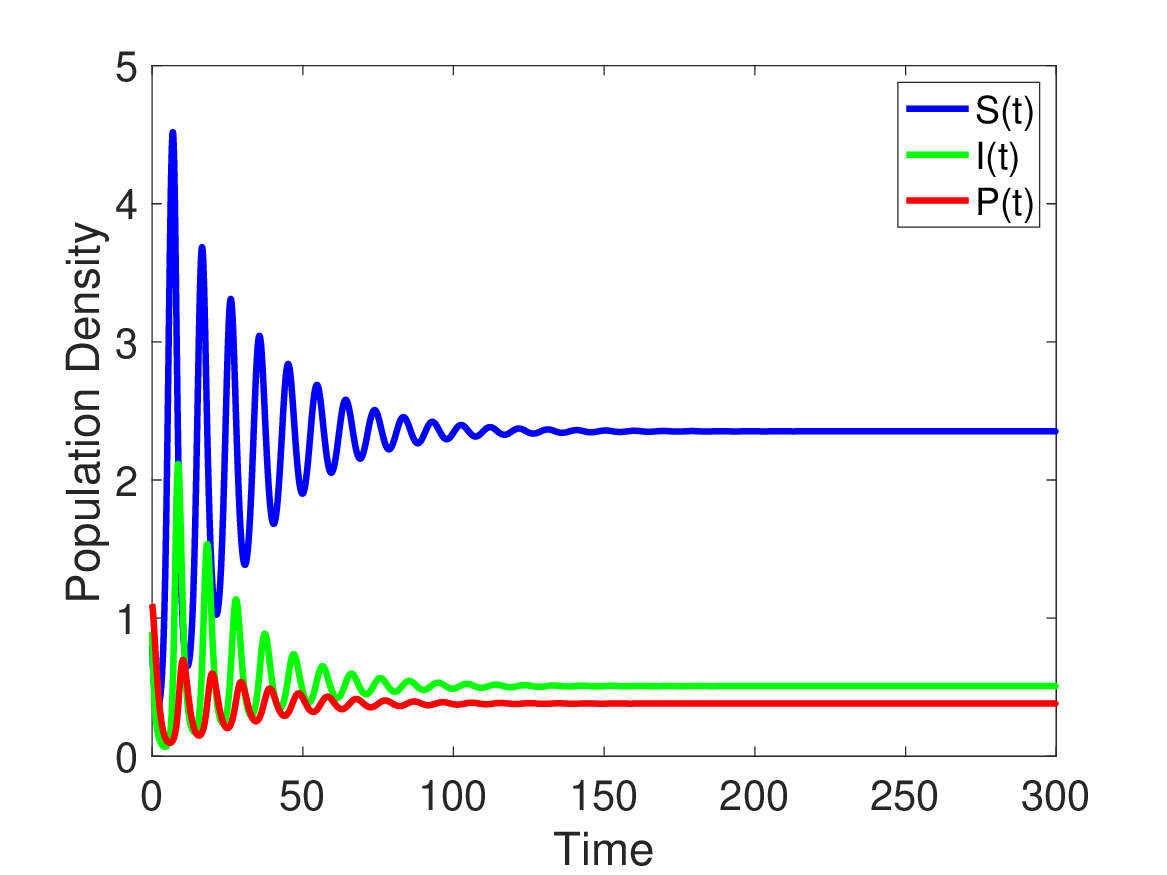}}
     \subfigure[]{
    \includegraphics[width=4cm, height=4cm]{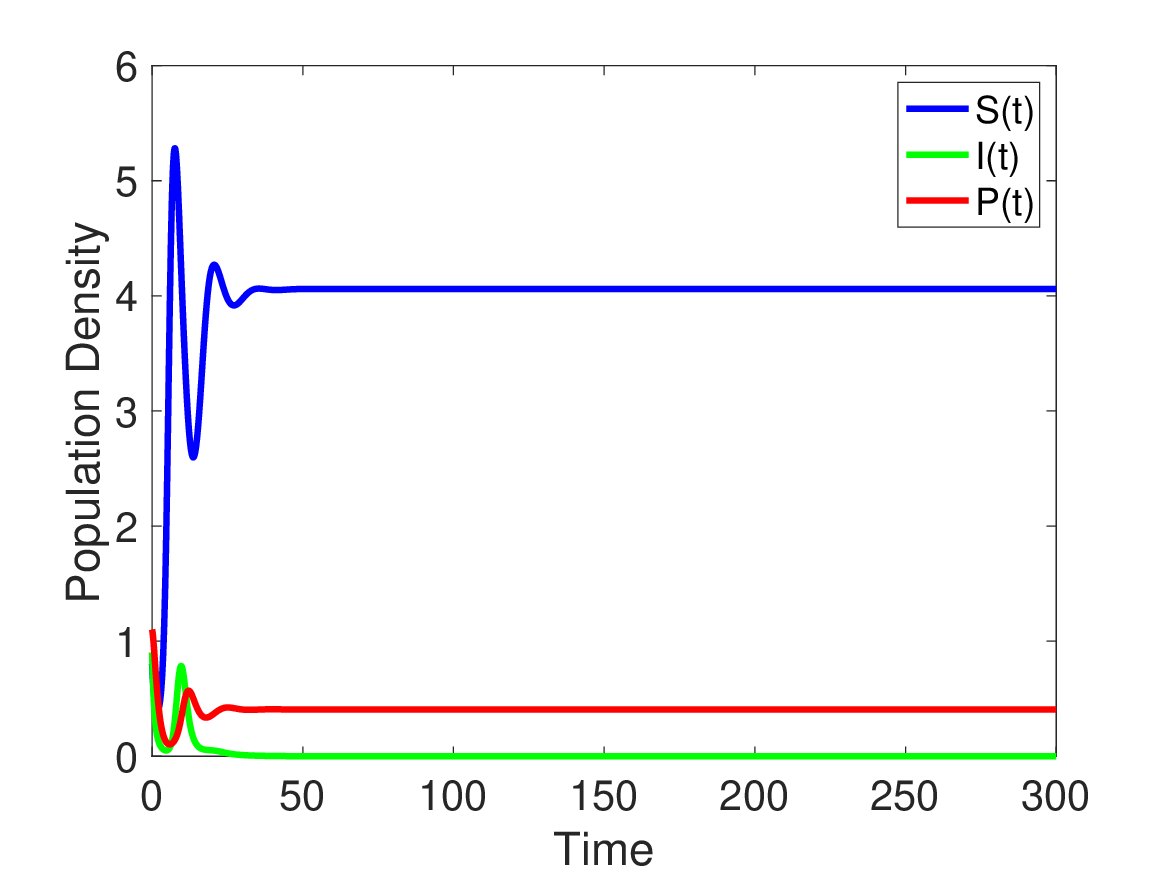}}
%%%%%%%%%%%%%
\subfigure[]{
    \includegraphics[width=4cm, height=4cm]{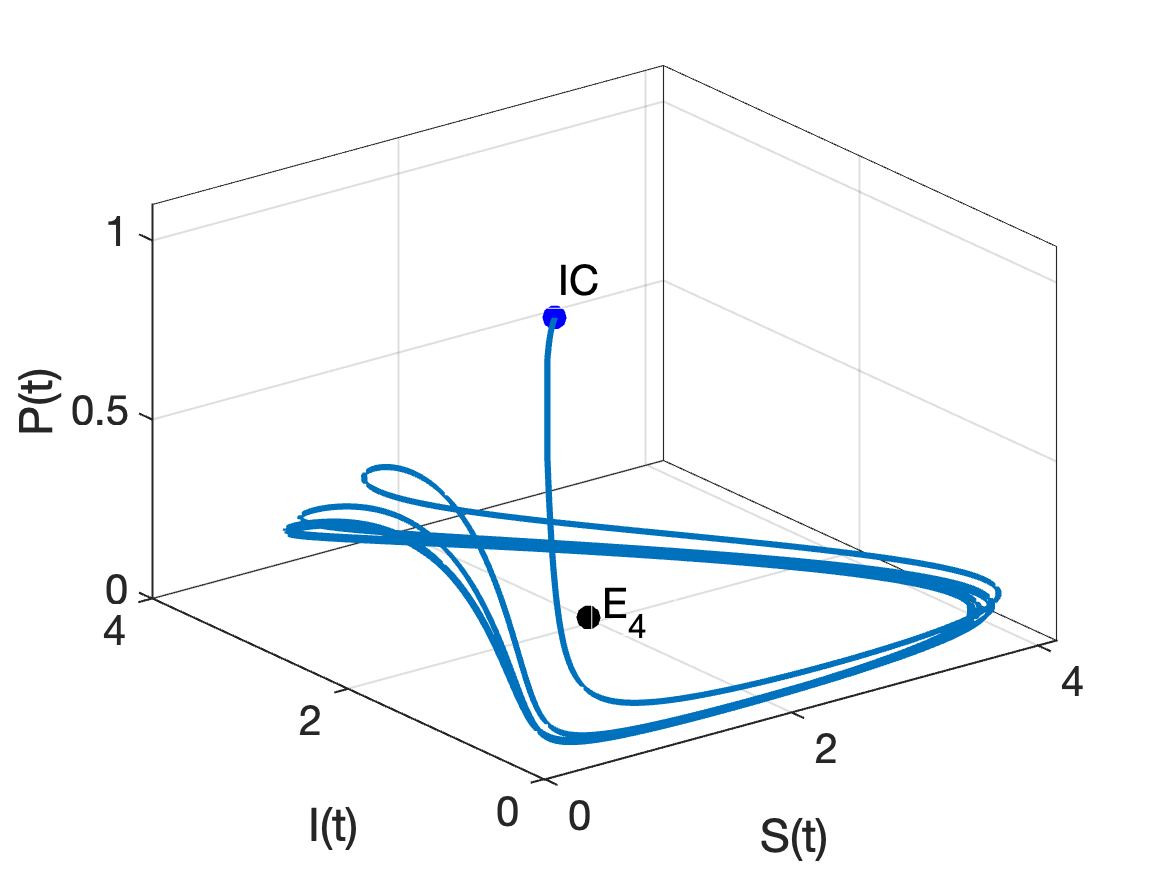}}
\subfigure[]{    
    \includegraphics[width=4cm, height=4cm]{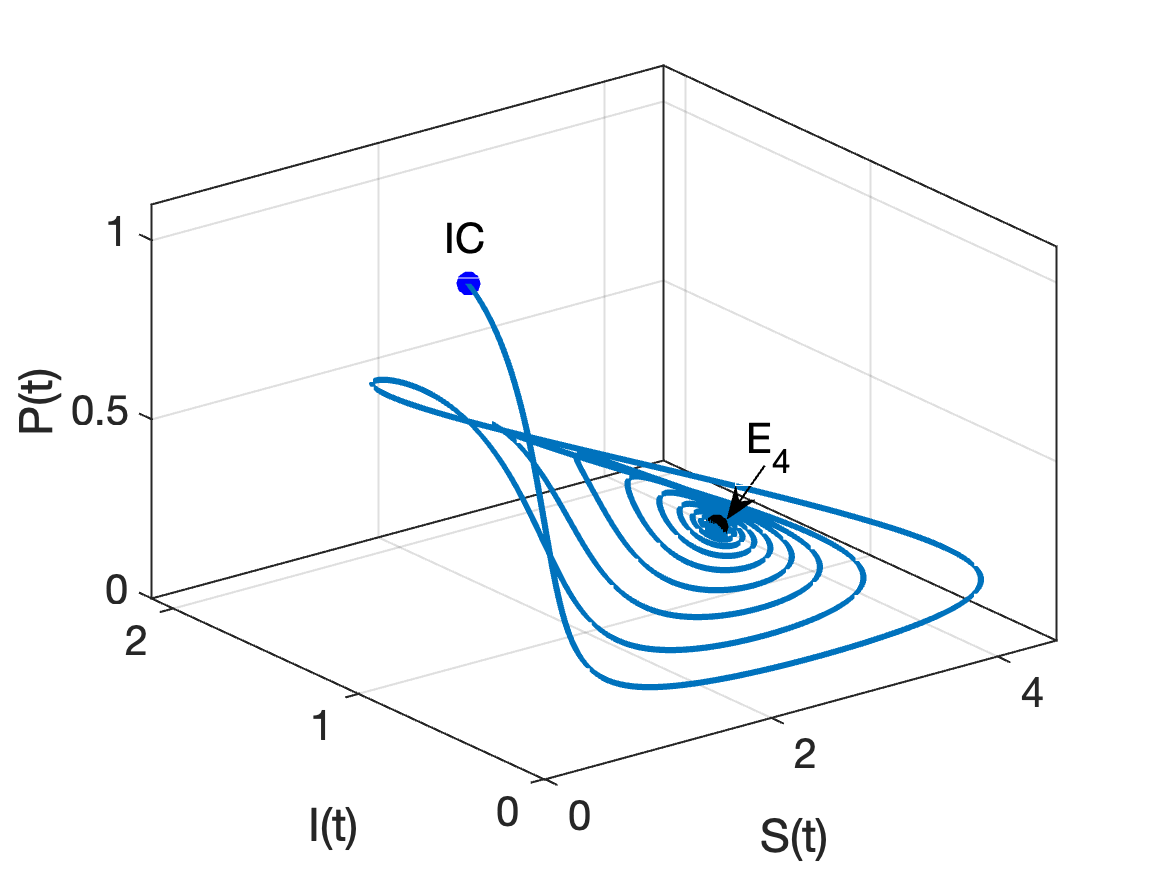}}
     \subfigure[]{
    \includegraphics[width=4cm, height=4cm]{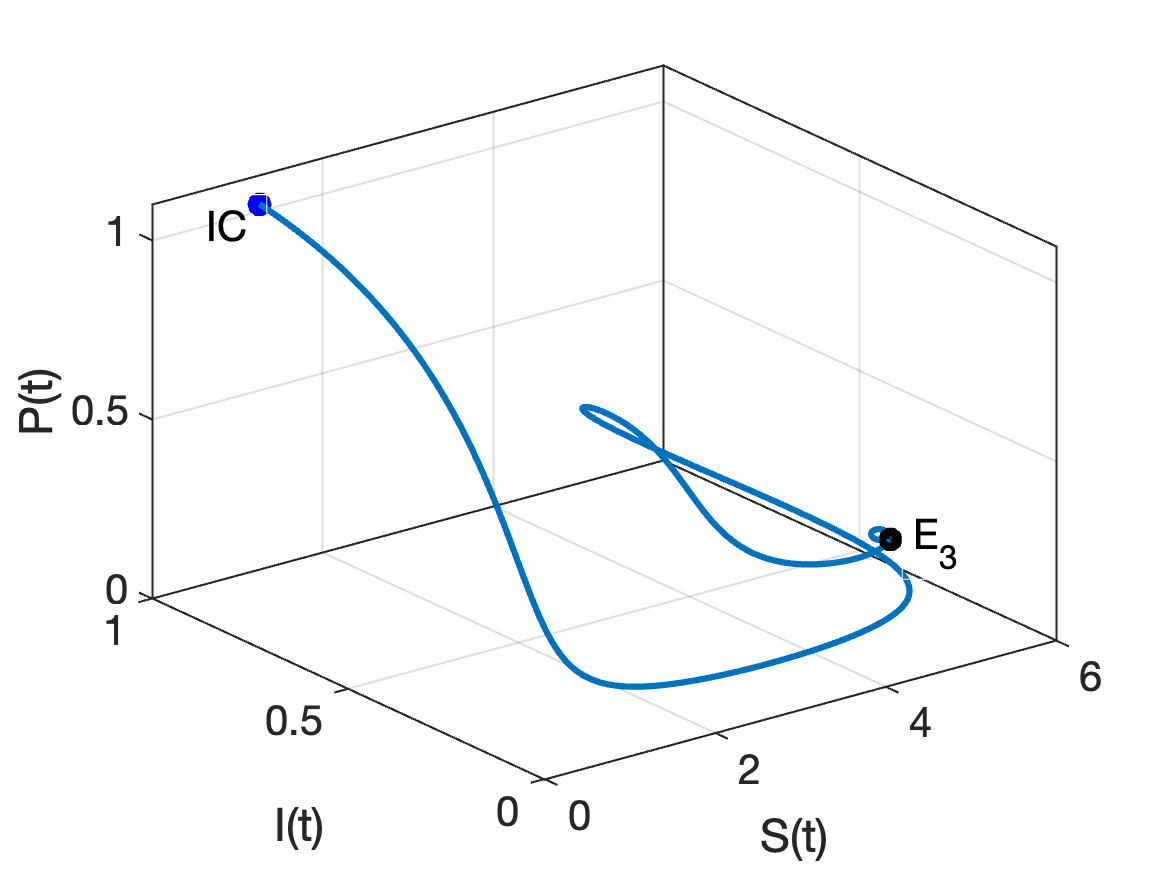}}
%%%%%%%%%%%%        
%\subfigure[]{    
%    \includegraphics[width=4cm, height=4cm]{bifHTCk2.eps}}
\end{center}
\caption{Figures depicting time evolution of population densities and phase portrait with  IC$=(0.8,0.9,1.1)$. ((a) \& (d))  At $k_2=0$, unstable $E_4=(1.1172,0.9516,0.2265)$, ((b) \& (e))  at $k_2=2$, stable $E_4=(2.3524,0.5080,0.3816)$ ((c) \& (f))  at $k_2=7$ stable $E_3=(4.0600,0,0.4070)$.  All other parameters are fixed and given as $b_0=2,k_1=2.8,~~K=8,~a_0=0.3,~d_0=0.6,~r=0.7,~e_0=0.5,~a_1=0.4,~d_1=0.7,~a_2=0.8,~d_2=0.3,~d_3=0.5$.}
\label{fig:bif:hopfTCk2}
\end{figure}

\subsection{Selective predation mediated disease control } 
In this study, we investigate the significant role of predators in regulating disease outbreaks within prey populations. Predators possess the ability to selectively target and consume infected individuals, leading to a higher mortality rate among the infected prey \cite{WA21}. This targeted predation has a direct impact on reducing the prevalence of the disease within the prey population, effectively mitigating the overall disease burden. The interaction between predators and prey becomes crucial as predators play a pivotal role in regulating the population sizes of their prey species. Understanding the interplay between predators and disease dynamics provides valuable insights into disease control strategies and the intricate balance of ecological systems. For further readings on selective predation, see \cite{C78, W21} and references therein.

\begin{definition}[Extinction of Infectious Prey]
Extinction of infectious prey refers to the complete disappearance or elimination of the prey population of individuals that are infected with a specific disease. Mathematically, extinction of a population (say the infectious population) can be defined as follows:
\[\lim_{t\rightarrow \infty}I(t)=0.\]
\end{definition}

\begin{theorem}\label{thm:sus-prey extinction}
For the SIP model described by the equations in model \eqref{Mainsystem}, under a specific parameter set and initial data $(S(0), I(0), P(0))$ that converges uniformly to a stable endemic state, there exists a value $d_1>\frac{e_0K - a_1b_0}{b_0}$ such that the solution for the infectious prey population starting from the same initial data will eventually go extinct, i.e. $\lim_{t \to \infty} I(t) = 0$.
\end{theorem}
\begin{proof}
We assume a particular parameter set and initial data $(S(0), I(0), P(0))$ such that they exhibit uniform convergence to a stable endemic state.
Consider the susceptible prey equation given in model \eqref{Mainsystem}, then
\begin{align}
\frac{dS}{dt}\leq b_0S\left(1 - \frac{S}{K}\right)
\end{align}
and $S(t)\leq\frac{K}{b_0}$ for all $t>0$. This is true via comparison with the logistic equation.
Next, we consider the predator equation given in the model \eqref{Mainsystem}. Via positivity
\begin{align}
\frac{dP}{dt}\geq -a_2P
\end{align}
we obtain
\begin{align}\label{cond:pgreater}
P(t)\geq P(0)e^{-a_2t}.
\end{align}
By substituting the upper bound of $S(t)$ and \eqref{cond:pgreater} into the equation governing the infectious prey population in model \eqref{Mainsystem}, we obtain
\begin{align}\label{cond:Iless}
\frac{dI}{dt} &\leq -a_1I+\dfrac{e_0K}{b_0}I-d_1P(0)e^{-a_2t}I \nonumber \\ 
&\leq -\left(a_1-\dfrac{e_0K}{b_0}+d_1P(0)e^{-a_2t}\right)I  
\end{align}
We observe that the differential equation given by
\begin{align}
\frac{dI}{dt} &\leq -\left(a_1-\dfrac{e_0K}{b_0}+d_1\right)I
\end{align}
implies that if the condition $d_1>\frac{e_0K-a_1b_0}{b_0}$ holds, the infectious prey population will inevitably go extinct. To ensure the eventual extinction of the infectious prey population, we want the right-hand side of the inequality in \eqref{cond:Iless} to be negative for all $t > 0$, thus
 $$\left(a_1-\frac{e_0K}{b_0}+d_1P(0)e^{-a_2t}\right)>0, \qquad \text{for all} \; t>0.$$ 
 We note that 
 \begin{equation*}
 \lim_{t\rightarrow \infty}d_1P(0)e^{-a_2t}=0.
 \end{equation*}
 
\noindent Therefore, to ensure that $d_1P(0)e^{-a_2t}$ remains positive for all $t > 0$, we need $P(0)e^{-a_2t} \gg \frac{1}{d_1}$. This condition implies that the initial population of predators, $P(0)$, should be sufficiently large. This ensures that the infectious prey population will eventually go extinct. 
\end{proof}

\begin{remark}
To investigate the impact of predator-prey interactions on disease dynamics, we manipulated the attack rate of the predator on the infectious prey, denoted as $d_1$. By systematically increasing $d_1$ from an initial value of $0.8$ to a higher value of $6$, we observed a significant transition in the model solutions. Specifically, in Figure \ref{fig:selective predation}, we visually captured the shift from a stable coexistence (or endemic) state at equilibrium $E_4$ to an infectious-free state at equilibrium $E_3$. 
\end{remark}

\begin{figure}[hbt!]
\begin{center}
  % \subfigure[]{
    \includegraphics[width=7.5cm, height=7cm]{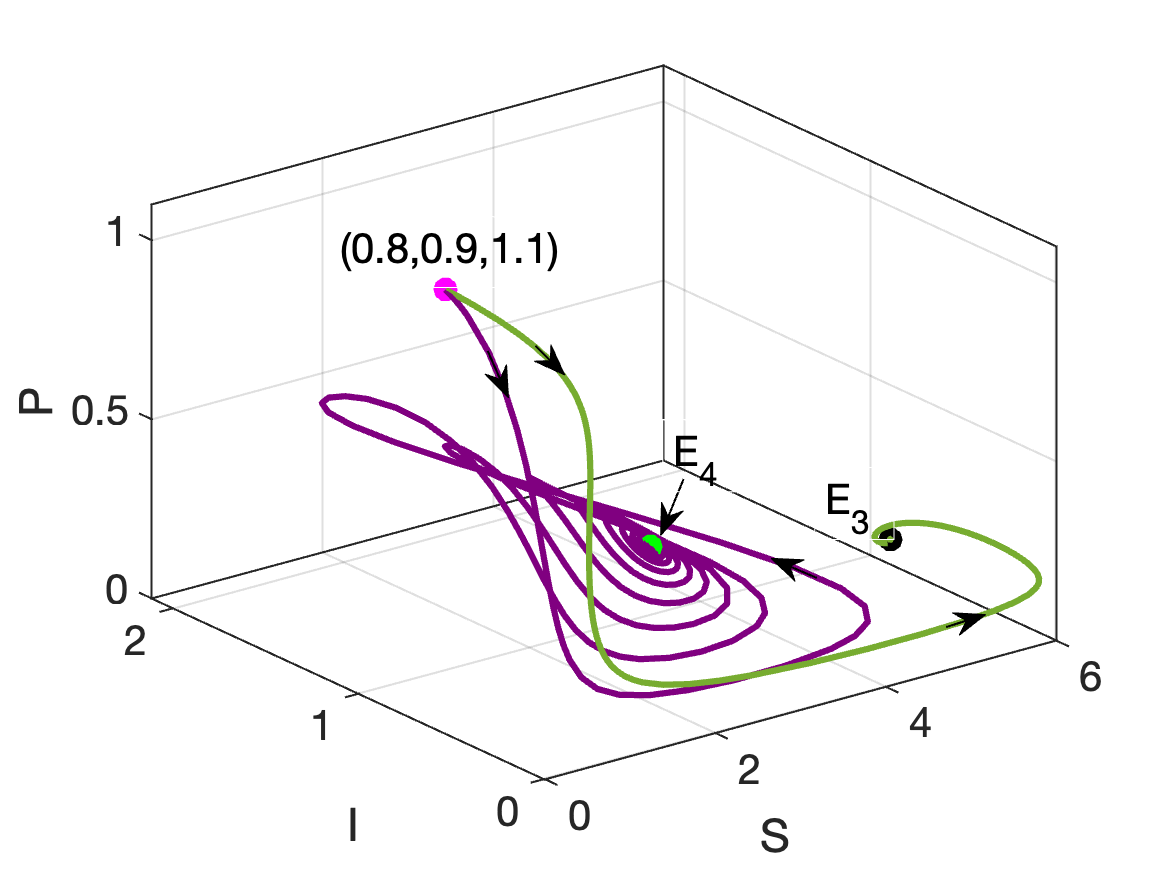}
\end{center}
\caption{Figure illustrating the effect of selective predation via the attack rate on infectious prey. Stable dynamics at $d_1=0.8$ for $E_4=(2.3492,0.5091,0.3758)$ and at $d_1=6$, $E_3=(4.0600,0,0.4070)$ is stable. All other parameters are fixed and given in the caption of Figure \ref{fig:bif:hopfTCk2}.}
\label{fig:selective predation}
\end{figure}

%%%%%%%%% Conclusion
\section{Conclusion}\label{sec:conclusion}
\noindent In this research, we proposed and analyzed an eco-epidemiological model incorporating dual fear effect, prey aggregation, infectious disease in prey. We demonstrated the well-behavedness of the model via nonnegativity and boundedness of solutions. We analyzed the existence of biologically significant equilibria within model \eqref{Mainsystem} and investigated their local stability. Also, we explored and provided mathematical proofs to some co-dimension one bifurcations including saddle-node, Hopf, and transcritical bifurcations. By investigating these bifurcations, researchers acquire valuable understanding regarding the emergence of intricate dynamics and may possess the ability to regulate and manipulate system behavior.

We obtained numerically an organizing center in co-dimension two where a loop is formed. This is illustrated in a 3-D plot shown in Figure \ref{fig:ZH-SNTC}(c) for the free parameters $k_2$ and $K$. Additionally, a zero-Hopf bifurcation occurred in co-dimension two as $k_2$ was varied in conjunction with $d_0$ or $K$, see Figure \ref{fig:ZH-SNTC}(a) and (b). We observed the existence of multiple zero-Hopf bifurcation in Figure \ref{fig:ZH-SNTC}(c). To the best of our knowledge, this is the inaugural work in the field of eco-epidemiology incorporating dual fear effects to demonstrate the occurrence of saddle-node transcritical and zero-Hopf bifurcations.

Finite time extinction of the susceptible prey population is possible with model \eqref{Mainsystem}. This phenomenon seen due to the singularity at the origin from the aggregation term ($S^r$ for $0<r<1$). This leads to non-uniqueness of solutions in backward time.  We observed a transition from a stable endemic state to an extinction of the susceptible population in finite time followed by the eventual extinction of the infectious and predator populations in infinite time. This is corroborated with Figure \ref{fig:FTE}.

Fear-based disease control involves utilizing fear as a motivator to implement measures and strategies aimed at preventing and controlling the spread of diseases. By instilling a sense of fear or concern about the consequences of the disease, populations are more likely to take necessary precautions and engage in behaviors that help mitigate the risk of infection and transmission. We observed via numerical simulation in Figure \ref{fig:bif:hopfTCk1} that the continuous decrease in fear of predators has a cascading effect in reducing the spread of disease. Also, we noticed that as the level of fear that lowers disease transmission increases, it leads to a decline of the infectious prey population and eventual extinction, see Figure \ref{fig:bif:hopfTCk2}.

Furthermore, we manipulate the attack rate of the predator on the infectious prey. A transition in the solutions of the model were observed as we continuously increase the predator attack rate on infectious prey. We observed a shift from a stable endemic state to an infectious-free state highlighting the significant role of increased predator attack in eradicating the disease. This is depicted in Figure \ref{fig:selective predation} under significant biological assumptions. We have provided clear mathematical evidence that selective predation can strongly impact disease transmission.  Our findings are in accordance with experimental results in ecology and evolution \cite{WA21}.

As an interesting future endeavor, the authors shall study and investigate the number of periodic orbits (or limit cycles) bifurcating from the model \eqref{Mainsystem} for the zero-Hopf equilibrium points as certain key parameter values are varied, see Figure \ref{fig:ZH-SNTC}. To achieve this, we will possibly employ some classical methods in dynamical systems such as the normal form theory \cite{ZP20} and time-averaging theory \cite{LM16, SV07}. The authors will in addition explore techniques needed to tentative prove conjectures \ref{conj:ZH k_2,b_0} - \ref{conj:SNTC k_2,K} in a forthcoming paper.

\section*{Acknowledgment}
\noindent KAF, SPW, and KHB would like to express their profound gratitude to CURM minigrant and NSF grant DMS-1722563.

\section*{Conflict of interest}
\noindent There is no conflict of interest.

\section*{Author's Contribution}
\noindent All authors contributed equally to this manuscript. All authors read and approved the final manuscript.

\section*{Availability of data and materials}
\noindent Not applicable.
%

%%%%%%%%%%%%%%%%%%%%%%%%%%%%%%%%%%%%%%%%%%
%\newpage %% ======= Reference

%\bibliographystyle{plain}
\bibliographystyle{unsrt}
\bibliography{manu1.bib}

%%%%%%%%%%%%%%%%%%%%%%%%%%%%%%%%%%%%%%%%%%
\end{document}